\documentclass[11pt]{article}
\usepackage[utf8]{inputenc}
\usepackage[english]{babel}
\usepackage{authblk}
\usepackage{graphicx}        
\usepackage{subfigure}                            
\usepackage{multicol}        
\usepackage[bottom]{footmisc}
\usepackage{placeins}
\usepackage{changes,todonotes}
\usepackage{pgf}
\usepackage{psfrag}
\usepackage{pgfplots}
\usepackage{hyperref}
\usepackage{appendix}
\usepackage{amsmath,amssymb}
\usepackage{bm}
\usepackage{mathtools}
\usepackage{color}
\usepackage{stmaryrd}
\usepackage{amsthm}
\usepackage{blindtext}
\usepackage{caption}
\usepackage[margin=3cm]{geometry}
\usepackage[ruled,vlined]{algorithm2e}

\graphicspath{{./fig/}}
\newcommand{\bp}{\bm{p}}

\newcommand{\V}{\mathbb{V}}
\newcommand{\TNORM}  [2]{\left|\nlen\left|\nlen\left|#1\right|\nlen\right|\nlen\right|_{{}_{#2}}}
\newcommand{\NORM}   [2]{\left|\nlen\left|#1\right|\nlen\right|_{#2}}
\newcommand{\nlen}{\hspace{-0.2mm}}

\newcommand{\TERM}[2]{\textsf{#1}_{#2}}

\newcommand{\bn}{\bm{n}}

\newcommand{\bfe}{\boldsymbol{f}}

\newtheorem{problem}{Problem}[section]
\newtheorem{theorem}{Theorem}[section]
\newtheorem{lemma}[theorem]{Lemma}
\newtheorem{proposition}[theorem]{Proposition}
\newtheorem{corollary}[theorem]{Corollary}
\newtheorem{remark}[theorem]{Remark}

\newcommand{\ilario}[1]{\textcolor{black}{#1}}
\newcommand{\toCheck}[1]{{\color{black} #1}}

\newcommand{\as}[1]{\textcolor{black}{#1}}
\newcommand{\fd}[1]{\textcolor{black}{#1}}

\begin{document}

\title{A Virtual Element Method for the wave equation\\ on curved edges in two dimensions}

\author[$\diamond$]{Franco~Dassi}
\author[$\star$]{Alessio~Fumagalli}
\author[$\star,\ddagger$]{Ilario~Mazzieri}
\author[$\star$]{Anna~Scotti}
\author[$\diamond$]{Giuseppe~Vacca}

\affil[$\diamond$]{Dipartimento di Matematica e Applicazioni, Universit\`a di Milano Bicocca, via R. Cozzi 55, 20125 Milano, Italy.}
\affil[$\star$]{MOX, Laboratory for Modeling and Scientific Computing, Dipartimento di Matematica, Politecnico di Milano, Piazza Leonardo da Vinci 32, I-20133 Milano, Italy.}
\affil[$\ddagger$]{Correspondig author.}

\affil[ ]{\texttt {franco.dassi@unimib.it, alessio.fumagalli@polimi.it, ilario.mazzieri@polimi.it, anna.scotti@polimi.it, giuseppe.vacca@unimib.it }}

\maketitle

\begin{abstract}
    In this work we present an extension of the Virtual Element Method with curved edges for the numerical approximation of the  second order wave equation in a bidimensional setting. Curved elements are used to describe the domain boundary, as well as internal interfaces corresponding to  the change of some mechanical parameters. As opposite to the classic and isoparametric Finite Element approaches, where the geometry of the domain is approximated respectively by piecewise straight lines and by higher order polynomial maps, in the proposed method the geometry is exactly represented, thus ensuring a highly accurate numerical solution. Indeed, if in the former approach the geometrical error might deteriorate the quality of the numerical solution, in the latter approach the curved interfaces/boundaries are approximated exactly guaranteeing the expected order of convergence for the numerical scheme. Theoretical results and numerical findings  confirm the validity of the proposed approach.
\end{abstract}



\section{Introduction}

In this paper we present an application of the  Virtual Element Method (VEM) with curved faces for the numerical solution of wave propagation problems.  Acoustics waves arise in many different scientific disciplines such as medical ultrasound, musical acoustics, vibro- and aero-acoustics, electromagnetics, geophysical exploration and seismology.  From a computational point of view these problems present
several challenges that reflect in the characteristics required by the
underlying numerical schemes such as \textit{geometrical flexibility},
\textit{high-accuracy} and \textit{scalability}.  \textit{Geometrical flexibility} is
important in order to have an optimal representation of the real geometry of the physical problem and \textit{high-accuracy} results without numerical artefacts, e.g., dispersion and dissipation errors, due to an improper model discretization. \textit{Scalable} and efficient algorithms are required to
solve realistic problems (involving typically milions of unknowns) and provide rapid feedback on the system status.

The scientific and technological progress involving the development of high-performance computing machines has made it possible to simulate, with increasing accuracy,
wave propagation phenomena for problems of a very complex nature. Nowadays, in computational acoustics, the most \as{widely} employed numerical techniques include the
Spectral Element (SE) \cite{Seriani1997,debasabe-sen,ZAMPIERI2006,Komatitsch2000}, the discontinuous Galerkin (dG) \cite{riviere2003discontinuous,GrScSc06,AntoniettiMazzieriMuhrNikolicWohlmuth_2019,terrana,Schwab2018}   and the
Finite Volumes (FV) \cite{dumbser-kaser2008,ghattas} schemes, typically built over unstructired grids composed by
tetrahedral/hexahedral elements in three
dimensions.   Although commercial software allows for the generation of
computational grids with complex domain geometry, this step can still represent a serious bottleneck for the entire simulation process. For this
reason, the use of general polygonal and polyhedral meshes is desirable.
Indeed, it is evident that with polytopal elements one can easily account for
small features in the model (such as cracks, holes and inclusions), \as{and handle in an automatic way} hanging nodes, movable meshes and adaptivity.

In the last decade, the development and analysis of numerical methods that
support computational meshes composed of polytopic elements have received a
lot of attention from the scientific community as testified by the \as{progress of the} Mimetic
Finite Difference (MFD) method \cite{BrLiSi2005,BrLiSh2006,BeLiMa2014,AntoniettiBeiraoBigoniVerani_2014,AnFoScVeNi2016} and the Virtual Element Method (VEM)
\cite{BeBrCaMaMaRu2012,BeiraoBrezziMariniRusso_2016,AnBeMoVe2014,AntoniettiManziniMazzieriMouradVerani_2019,BenedettoBerronePieracciniScialo_2014,AnBeScVe2016,BeiraoDassiRusso_2017,BeiraoBrezziDassiMarini_2018,BeiraoRussoVacca_2019,BeiraoMoraVacca_2019}.
 in the conforming setting or by the Discontinuous Galerkin (DG) methods
\cite{AntoniettiHoustonHuSartiVerani_2017,AntoniettiMazzieri_2018,AnHoPe2018,AnBoMa2020,Antonietti_et_al_2015,CangianiDongGeorgoulisHouston_2017},the Hybrid High-Order (HHO) method
\cite{DiPietroErn_2015,AghiliDiPietroRuffini_2017,BottiDiPietroSochala_2017,BottiDiPietro_2018,ChaveDiPietroFormaggia_2018,DiPietroKrell_2018}, the Gradient Schemes \cite{DrEyHe2016}
and the non-confroming VEM \cite{AnMaVe2017,AyLiMa2016,CaMaSu2017}, in the non-conforming setting.

With few exceptions, e.g., {\cite{BrLiSh2006,BottiDiPietro_2018,BeiraoRussoVacca_2019,BPP:2019,BBMR:2020,ABD:2020,BCDE:2021}}, those methods make use of polygonal and polyhedral
meshes with straight edges and faces that, especially for high-order methods,
can deteriorate the accuracy of the solution \as{in the case of curved boundaries or interfaces}. Indeed, as it is known from the
FEM literature, the approximation of the domain geometry with \as{planar} facets
introduces an error that can dominate the analysis.  A better description of
the domain of interest can be obtained through high-order polynomial maps and
isoparametric FEM, while the exact representation of computational (CAD) domains
is possible thanks to the Iso-Geometric Analysis (IGA). Indeed, in the latter,
the same spline maps are employed for the parametrization of the geometry and
the problem solution \cite{Huges2009}.
As it is shown by the seminal paper \cite{BeiraoRussoVacca_2019} and in \cite{dassi2020mixed} for the Darcy problem,
through the VEM technology it is possible to define discrete space also on curved elements in such a way that the domain geometry is defined exactly.
Indeed, by exploiting the peculiar construction of the VEM, one can avoid not only the approximation (even with polynomial functions) of the domain but also the positioning of the isoparametric nodes \cite{Lenoir1986}. Moreover, \as{only} the \as{local}  parametrization of the cells boundary is needed as opposite to the IGA where also the internal elemental volume has to be considered. On the other side, since the construction of the VEM space in directly made on the physical mesh elements,  the application of the VEM on curved geometry is computationally more expensive with respect to isoparametric FEM of IGA.

In this work, we apply the VEM with curved elements for the simulation of the acoustic wave propagation problem. The formulation is obtained starting from the seminal paper  \cite{BeiraoRussoVacca_2019}.  To the best of the authors' knowledge this is the first time that such an approach is applied to the second order wave equation.

The rest of the paper is organized as follows. Section~\ref{Sc:ModelPb} defines the model setup, the main assumptions on the curved domains and the VEM discretization.
\as{ In Section~\ref{Sc:TeorhericalAnalysis} we derive the theoretical analysis of the method.} 
In  Section~\ref{Sc:NumericalResults} we present some verification tests  assessing the accuracy of the method and some applications of the proposed method to realistic scenarios. 
Finally, we draw our conclusions in Section~\ref{Sc:Conclusions}.

\paragraph{Notation.} Given a domain $A\subset \mathbb{R}^2$, we consider the space $L^2(A)$ to be the
classical space of functions which are squared measurable $L^2(A)= \{v: \int_A
\vert v \vert^2 < +\infty \}$. Its associated scalar
product and induced norm are given by: $(\cdot, \cdot)_A: L^2(A) \times L^2(A)
\rightarrow \mathbb{R}$ and $\Vert \cdot \Vert_A : L^2(A) \rightarrow \mathbb{R}$
and defined as
$  (p, v)_A = \int_A p v$ and $\Vert p \Vert_A = \sqrt{(p, p)_A}$.
In the case of $L^2$-vector valued functions the extension is trivial and we
indicate with $[L^2(A)]^2$ such space.

We consider also the Sobolev space $H^1(A) = \{v\in L^2(A): \nabla v \in
[L^2(A)]^2\}$ with semi-norm $\vert \cdot \vert_{H^1(A)}: H^1(A) \times H^1(A)
\rightarrow \mathbb{R}$ and norm $\Vert \cdot \Vert_{H^1(A)}: H^1(A) \times H^1(A)
\rightarrow \mathbb{R}$ given by $ \vert p \vert_{H^1(A)} = \Vert \nabla p \Vert_A$ and $    \Vert p \Vert_{H^1(A)} = \sqrt{\Vert p \Vert_A^2 + \vert p \vert_{H^1(A)}^2}$.
We indicate with $H^1_\Upsilon(A)$ the subspace of $H^1(A)$ such that the
functions are null on $\Upsilon \subset \partial A$.

Since we are dealing with a time dependent problem, we will also consider the
following Bochner spaces. By considering a scalar $T>0$, an integer $1 \leq p <
\infty$, and a generic functional space $X$, we denote by $L^p((0, T]; X)$ as
the space of function $v:(0,T)\rightarrow X$  such that $v$ is measurable and
$\int_{0}^{T} \Vert v(t) \Vert_X^p  dt < +\infty$. The spaces $C^n((0, T]; X)$
with $0\leq n\leq \infty$ are defined in a similar way. The time derivative will
be indicated with a dot, i.e., 
we exploit the following notation $\dot{p} = \frac{\partial p}{\partial t}$.

\section{Model problem and its Virtual Element Discretization}\label{Sc:ModelPb}

Let $\Omega \subset \mathbb{R}^2$  be an open bounded domain with
regular boundary $\Gamma$ having outward pointing unit normal $\bn$, and set $T>0$. The
mathematical model of acoustic wave propagation can be formulated in the
following problem.
\begin{problem}[Wave problem - strong formulation]\label{modelpb}
    Find
    $p:\Omega \times (0,T] \rightarrow \mathbb{R}$ such that:
    \begin{gather*}
        \begin{aligned}
            &\rho \ddot{p} - \nabla \cdot (\mu \nabla p) = f && {\rm in} \, \Omega \times (0,T],\\
            &p = 0 && {\rm on} \, \Gamma_D \times (0,T],\\
            &\mu \nabla p \cdot \bn = 0  && {\rm on} \, \Gamma_N \times (0,T], \\
            &\mu \nabla p \cdot \bn +  \rho \dot{p}  = 0 && {\rm on} \, \Gamma_{A} \times (0,T], \\
            &(p, \dot{p}) = (p_0, p_1) && {\rm in} \, \Omega \times \{0\},
         \end{aligned}
     \end{gather*}
     where $\rho$ and $\mu$ are two positive \textcolor{black}{uniformly bounded} functions, representing
     the mass density and the viscosity of the medium, respectively.
\end{problem}
We assume the
boundary $\Gamma$ to be Lipschitz and to be decomposed into non-overlapping sufficiently smooth curves $\Gamma_D$,
$\Gamma_N$ and $\Gamma_A$ such that $\Gamma = \Gamma_D \cup \Gamma_N \cup \Gamma_A$. On $\Gamma_D$
(\textit{soft sound boundary}) the pressure field is set equal to zero, on
$\Gamma_N$ (\textit{sound hard boundary}) a rigid wall condition is imposed and on $\Gamma_A$ (\textit{absorbing boundary}) a non-reflecting
condition is considered.
%
%

To derive the weak formulation, we set $V = H^1_{\Gamma_D}(\Omega)$ and we introduce the following bilinear forms
\begin{gather}\label{def:bil}
    \begin{aligned}
        & m: V\times V \rightarrow \mathbb{R} &&
        m(p,\ilario{v}) = (\rho p, v)_\Omega \quad \forall u, v \in V\\
        & a: V\times V \rightarrow \mathbb{R} &&
        a(p,v) = (\mu \nabla p, \nabla v)_\Omega \quad \forall p, v \in V\\
        & c: V\times V \rightarrow \mathbb{R} &&
        c(p,v) = (\rho p, v)_{\Gamma_{A}} \quad \forall p, v \in V
    \end{aligned},
\end{gather}
and the linear functional $F : V \rightarrow \mathbb{R}$ as $F(v) = ( f,
v)_\Omega$ for any $v\in V$.
\begin{problem}[Wave problem - weak formulation]\label{eq:weak_form}
    The weak formulation of Problem \ref{modelpb} is:
    for any time $t\in (0,T]$ find $p = p(t) \in V$ such that
    \begin{gather*}
        \begin{aligned}
            &m (\ddot{p}, v) + c(\dot{p}, v) + a( p,v ) = F( v )  && \forall v \in H^1_{\Gamma_D}(\Omega),\\
            &(p(0),\dot{p}(0)) = ( p_0, p_1),
        \end{aligned}
    \end{gather*}
\end{problem}
\textcolor{black}{By using standard arguments, cf. \cite{DuvantLions,Gander2005}, it can be proved that if $(p_0, p_1) \in V \times L^2(\Omega)$ and $f \in L^2((0,T]; L^2(\Omega))$, then Problem \ref{eq:weak_form} admits a unique solution $p \in C^0 ((0, T]; V) \cap C^1 ((0, T]; L^2 (\Omega))$}.

\subsection{Virtual Element Discretization on curved edges}\label{SubSc:Vemdis}

In this part we present how to approximate, with the Virtual Element Method,
Problem \eqref{eq:weak_form}. The main difference from a classical VEM formulation is the presence of
curved interfaces. For this reason we consider the approach first introduced in \cite{BeiraoRussoVacca_2019}
and then extended in \cite{dassi2020mixed} for scalar problems in mixed form.

Following \cite{BeiraoRussoVacca_2019}, we consider a sequence of computational tessellation $\Omega_h$ of the domain of
interest $\Omega$ into general polygons \toCheck{(having possibly curved interfaces)} indicated with $E \in \Omega_h$.  Clearly for $E, E^\prime \in \Omega_h$ such that $E \neq E^\prime$ we have $E \cap E^\prime = \emptyset$ and $\overline{\Omega_h} = \cup_{E \in
\Omega_h} \overline{E}$.
We let $$h_E = {\rm diameter}(E), \quad\quad  h = \sup_{E \in \Omega_h} h_E,$$ and suppose that for all $h$, each element $E\in\Omega_h$ fulfils the following assumptions:
\begin{itemize}
\item[(\textbf{A1})] $E$ is star shaped with respect ta a ball $B_E$ of radius $\geq \varrho h_E$,
\item[(\textbf{A2})] the length of any (possibly curved) edge of $E$ is  $\geq \varrho h_E$,
\end{itemize}
where $\varrho$ is a positive constant. An element $E$ has boundary $\partial E$ represented by a finite number of edges  $e \in \partial E$. The set of edges of a tesselation $\Omega_h$ is indicated with $\mathcal{E}_h$.

We assume that:
\begin{itemize}
\item[(\textbf{A3})] each curve $\Gamma_{i}$, for $i=A,D,N$, composing the boundary $\Gamma$ is of class $C^{m+1}$, with $m\geq 0$, such that it exists, for each of them, an invertible and regular map $\gamma_i: [a_i, b_i] \rightarrow \Gamma_i$ with $a_i, b_i \in \mathbb{R}$.
\end{itemize}

\toCheck{Additional internal curved interfaces $\Gamma_i$, for  $i=1,...,n$ (cf. Figure~\ref{fig:domain_curv}) representing a sharp variation in the mechanical parameters, i.e., $\mu$ and $\rho$, verify assumption \textbf{A3}.}
{ For simplicity, in the following, we assume that $\rho$ and $\mu$ in Problem \ref{modelpb} are piecewise constants with respect to the decomposition $\Omega_h$.}

In the case of a single curved boundary/interface, to ease the presentation we will drop the subscript $i$. At the grid level,  elements facing  $\Gamma_i$ have
curved edges. See Figure \ref{fig:domain_curv} as an example.
\begin{figure}[htb]
    \centering
    \resizebox{0.33\textwidth}{!}{\fontsize{0.75cm}{2cm}\selectfont
\begingroup%
  \makeatletter%
  \providecommand\color[2][]{%
    \errmessage{(Inkscape) Color is used for the text in Inkscape, but the package 'color.sty' is not loaded}%
    \renewcommand\color[2][]{}%
  }%
  \providecommand\transparent[1]{%
    \errmessage{(Inkscape) Transparency is used (non-zero) for the text in Inkscape, but the package 'transparent.sty' is not loaded}%
    \renewcommand\transparent[1]{}%
  }%
  \providecommand\rotatebox[2]{#2}%
  \newcommand*\fsize{\dimexpr\f@size pt\relax}%
  \newcommand*\lineheight[1]{\fontsize{\fsize}{#1\fsize}\selectfont}%
  \ifx\svgwidth\undefined%
    \setlength{\unitlength}{183.0867473bp}%
    \ifx\svgscale\undefined%
      \relax%
    \else%
      \setlength{\unitlength}{\unitlength * \real{\svgscale}}%
    \fi%
  \else%
    \setlength{\unitlength}{\svgwidth}%
  \fi%
  \global\let\svgwidth\undefined%
  \global\let\svgscale\undefined%
  \makeatother%
  \begin{picture}(1,0.50086814)%
    \lineheight{1}%
    \setlength\tabcolsep{0pt}%
    \put(0,0){\includegraphics[width=\unitlength,page=1]{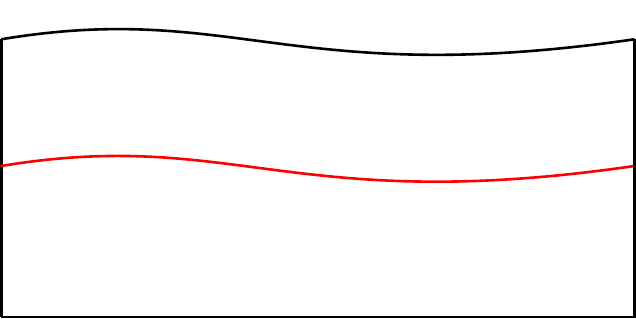}}%
    \put(0.45761671,0.44692129){\color[rgb]{0,0,0}\makebox(0,0)[lt]{\lineheight{1.25}\smash{\begin{tabular}[t]{l}$\Gamma_1$\end{tabular}}}}%
    \put(0.45248223,0.25087985){\color[rgb]{0,0,0}\makebox(0,0)[lt]{\lineheight{1.25}\smash{\begin{tabular}[t]{l}$\Gamma_0$\end{tabular}}}}%
    \put(0.05622994,0.1220877){\color[rgb]{0,0,0}\makebox(0,0)[lt]{\lineheight{1.25}\smash{\begin{tabular}[t]{l}$\Omega$\end{tabular}}}}%
  \end{picture}%
\endgroup%
}%
    \hspace*{0.1\textwidth}%
    \resizebox{0.45\textwidth}{!}{\fontsize{0.75cm}{2cm}\selectfont
\begingroup%
  \makeatletter%
  \providecommand\color[2][]{%
    \errmessage{(Inkscape) Color is used for the text in Inkscape, but the package 'color.sty' is not loaded}%
    \renewcommand\color[2][]{}%
  }%
  \providecommand\transparent[1]{%
    \errmessage{(Inkscape) Transparency is used (non-zero) for the text in Inkscape, but the package 'transparent.sty' is not loaded}%
    \renewcommand\transparent[1]{}%
  }%
  \providecommand\rotatebox[2]{#2}%
  \newcommand*\fsize{\dimexpr\f@size pt\relax}%
  \newcommand*\lineheight[1]{\fontsize{\fsize}{#1\fsize}\selectfont}%
  \ifx\svgwidth\undefined%
    \setlength{\unitlength}{296.72889289bp}%
    \ifx\svgscale\undefined%
      \relax%
    \else%
      \setlength{\unitlength}{\unitlength * \real{\svgscale}}%
    \fi%
  \else%
    \setlength{\unitlength}{\svgwidth}%
  \fi%
  \global\let\svgwidth\undefined%
  \global\let\svgscale\undefined%
  \makeatother%
  \begin{picture}(1,0.29929045)%
    \lineheight{1}%
    \setlength\tabcolsep{0pt}%
    \put(0,0){\includegraphics[width=\unitlength,page=1]{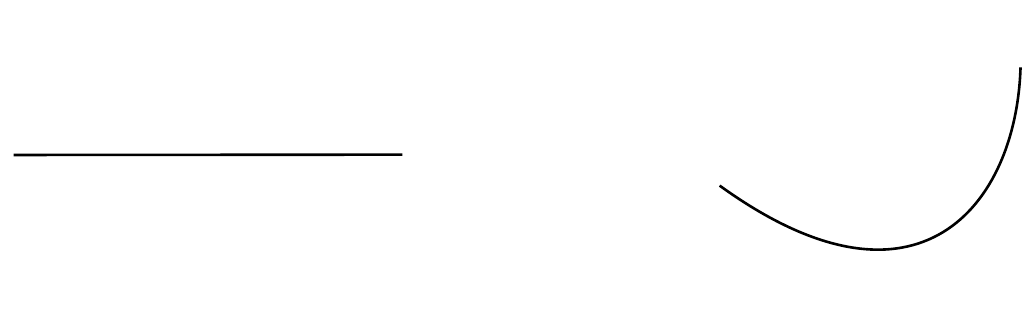}}%
    \put(0.14195419,0.17400083){\color[rgb]{0,0,0}\makebox(0,0)[lt]{\lineheight{1.25}\smash{\begin{tabular}[t]{l}$\mathfrak{e}$\end{tabular}}}}%
    \put(0.78265616,0.08012532){\color[rgb]{0,0,0}\makebox(0,0)[lt]{\lineheight{1.25}\smash{\begin{tabular}[t]{l}${e}$\end{tabular}}}}%
    \put(0,0){\includegraphics[width=\unitlength,page=2]{map_edge.pdf}}%
    \put(0.5242147,0.26600433){\color[rgb]{0,0,0}\makebox(0,0)[lt]{\lineheight{1.25}\smash{\begin{tabular}[t]{l}$\gamma$\end{tabular}}}}%
    \put(0.52730376,0.00911305){\color[rgb]{0,0,0}\makebox(0,0)[lt]{\lineheight{1.25}\smash{\begin{tabular}[t]{l}$\gamma^{-1}$\end{tabular}}}}%
    \put(0,0){\includegraphics[width=\unitlength,page=3]{map_edge.pdf}}%
  \end{picture}%
\endgroup%
}%
    \caption{On the left an example of a domain $\Omega$ with an internal
    interface $\Gamma_0$ and with a portion of the
    boundary curved $\Gamma_1$. On the right, the direct and inverse mapping
    between a curved edge $e$ and the reference interval $\mathfrak{e}$.}%
    \label{fig:domain_curv}
\end{figure}
In this case, and with an abuse
in notation, we still make use of the mapping $\gamma_i$ and extend it also for the
straight case. For a curved edge $e$, we have $\gamma_i : \mathfrak{e}
\rightarrow e$ where $\mathfrak{e} = [0, h_e]$ is a rectified reference segment,
while for a non curved edge $e$, $\gamma_i: \mathfrak{e} \rightarrow e$ is an affine map.

{
The proper characterization of the virtual element space goes along the following steps: 
\textit{(i)} the introduction of the local virtual element space;
\textit{(ii)} the selection of a number of  degrees of freedom that uniquely characterizes the virtual element functions of the local  space;
\textit{(iii)} the definition of projectors onto subspaces of polynomials that are computable by the degrees of freedom. 
}

\paragraph{Polynomial approximation spaces.}
{ For any integer $n \geq -1 $ and any element $E \in \Omega_h$}, we define
$\mathbb{P}_n(E)$ to be the set of polynomials on $E$ of degree less or equal to $n$. In the case $n=-1$ we set $\mathbb{P}_{-1}(E) = \{ 0 \}$.
Moreover, we introduce the global polynomial space as
\begin{gather*}
    \mathbb{P}_n(\Omega_h) = \{v \in L^2(\Omega_h): v|_E \in \mathbb{P}_n(E) \,
    \forall E \in \Omega_h\}.
\end{gather*}
Identifying with $\bm{x}_E$ and $h_E$ the centre and the diameter  of the element $E$, respectively, we introduce the space of normalized monomials as
\begin{gather*}
    \mathcal{M}_n(E) = \left\{ m \in L^2(E):\, m(\bm{x}) = \left( \frac{\bm{x} -
    \bm{x}_E}{h_E}
    \right)^{\bm{\beta}} \text{ for } \vert{\bm{\beta}}\vert\leq n\right\}.
\end{gather*}
The space $\mathcal{M}_n(E)$ forms a basis for $\mathbb{P}_n(E)$. For the edges
of the grid, we introduce approximation spaces that consider the curved
geometry. For a reference (rectified) segment $\mathfrak{e}$, we introduce the monomial set
\begin{gather*}
    \mathcal{M}_n(\mathfrak{e}) = \left\{m \in L^2(\mathfrak{e}):\, m(x) =  \left( \frac{x -
    x_{\mathfrak{e}}}{h_{\mathfrak{e}}}
    \right)^{\beta} \text{ for } \beta\leq n\right\},
\end{gather*}
with $x_{\mathfrak{e}}$ the midpoint of $\mathfrak{e}$ and $h_{\mathfrak{e}} =
\vert \mathfrak{e} \vert$ its size. Next, we define the mapped polynomial spaces on the edges in $\mathcal{E}_h$, given by
\begin{gather*}
    \widetilde{\mathbb{P}}_n(e) = \{ \widetilde{v} = v \circ \gamma^{-1}:\, v
    \in \mathbb{P}_n(\mathfrak{e}) \}
    \quad \text{and} \quad
    \widetilde{\mathcal{M}}_n(e) = \{ \widetilde{m} = m \circ \gamma^{-1}:\, m
    \in \mathcal{M}_n(\mathfrak{e}) \},
\end{gather*}
where $\gamma$ represents the local map of the edge $e$ to $\mathfrak{e}$ as discussed before.

\paragraph{Projection operators.}
 As a second step we introduce in this part the projection operators that are useful for the actual computation of the virtual element formulation given in the sequel. We firstly consider the  projector $\Pi_n^\nabla : H^1(E) \rightarrow \mathbb{P}_n(E)$ defined as
\begin{gather}\label{eq:projection_h1}
    \begin{aligned}
        & (\nabla \Pi_n^\nabla v, \nabla q_n)_E = (\nabla v, \nabla q_n)_E &&
        \forall \, q_n \in \mathbb{P}_n(E) \text{ and } \forall\, v \in
        H^1(E),\\
        & (\Pi^\nabla_n v, 1)_{\partial E} = (v, 1)_{\partial E} && \forall\, v
        \in H^1(E),
    \end{aligned}
\end{gather}
and secondly the $L^2$ projection operator $\Pi^0_n: L^2(E) \rightarrow
\mathbb{P}_n(E)$ which is given by
\begin{gather}\label{eq:projection_l2}
    (\Pi^0_n v, q_n)_E = (v, q_n)_E \quad \forall \, q_n \in \mathbb{P}_n(E) \text{ and } \forall\, v \in
    L^2(E).
\end{gather}
Finally,
we introduce the $L^2$ projection operator of vector valued functions defined as
$\bm{\Pi}_{n}^0: [L^2(E)]^2 \rightarrow [\mathbb{P}_{n}(E)]^2$ and given by
\begin{gather}\label{eq:projection_l2_grad}
    (\bm{\Pi}_{n}^0 \bm{v}, \bm{q}_n)_E = (\bm{v}, \bm{q}_n)_E \quad \forall\,
    \bm{v} \in [L^2(E)]^2 \text{ and } \forall\,
    \bm{q}_n \in [\mathbb{P}_{n}(E)]^2,
\end{gather}
which will be employed to approximate the gradient of $v$.

\paragraph{Approximation spaces and degrees of freedom.}
{ Let $k \geq 1$ be the polynomial order of the method.}
By following the approach  derived  in~\cite{BeiraoRussoVacca_2019} for elliptic problems, we select the following enhanced local virtual element space  defined as
\begin{gather}\label{eq:vk}
    \begin{aligned}
        V_k(E) = \{ v \in H^1(E):\, v \in C(\partial E), \, \Delta v \in \mathbb{P}_{k}(E), \, v \in
        \widetilde{\mathbb{P}}_k(e) \, \forall e \in \partial E, \\ (\Pi^\nabla_k v
        - v, q) = 0 \, \forall q \in \mathbb{P}_k(E) \setminus \mathbb{P}_{k-2}(E) \}.
    \end{aligned}
\end{gather}
We remark that, if $E$ is an element with only straight edges, then \eqref{eq:vk} is equivalent
to the { enhanced} VEM space as in \cite{projectors,BeiraodaVeiga2014a}.  In general, the space $V_k(E)$ does not give a closed form for computing its shape functions. 

{
We here summarize the main properties of the space $V_k(E)$ (we refer to \cite{BeiraoRussoVacca_2019,projectors} for a deeper analysis).

\begin{itemize}
\item [\textbf{(P1)}] \textbf{Polynomial inclusion:} $\mathbb{P}_0(E) \subseteq V_k(E)$ but in general $\mathbb{P}_k(E) \nsubseteq V_k(E)$.

\item [\textbf{(P2)}] \textbf{Degrees of freedom:}
the following linear operators constitute a set of DoFs for $V_k(E)$: for any $v \in V_k(E)$ we consider
\begin{itemize}
    \item the value of $v$ at the vertices of $E$;
    \item the values of $v$ mapped through $\gamma$ at the $k-1$ internal points of the Gauss-Lobatto quadrature rule with $k+1$ points;
    \item the internal scaled moments of $v$, up to order $k-2$, given by
    $\vert{E}\vert^{-1} (v, m_i)_E$ for any $m_i \in \mathcal{M}_{k-2}(E)$.
\end{itemize}

\item [\textbf{(P3)}] \textbf{Polynomial projections:}
the DoFs  allow us to compute the following linear operators:
\[
\Pi_k^{\nabla} \colon V_k(E) \to \mathbb{P}_k(E), \qquad
\Pi_k^0 \colon V_k(E) \to \mathbb{P}_k, \qquad
\bm{\Pi}_{k-1}^0 \colon \nabla V_k(E) \to [\mathbb{P}_{k-1}(E)]^2 \,.
\]
\end{itemize}
The global virtual element space is obtained by gluing such local spaces, i.e.
\begin{equation}
\label{eq:global}
V_k(\Omega_h) = \{v \in V \quad \text{s.t.} \quad v \in V_k(E) \quad \text{for any $E\in\Omega_h$} \}
\end{equation}
with the associated set of degrees of freedom.
}

\paragraph{The Virtual Element formulation.}
We introduce in this part the discrete weak formulation of \eqref{eq:weak_form}, by using
the projection operators and the functional spaces previously given.
Given an element $E$, by considering the trial and test functions in the space $V_k(E)$ both forms $m(\cdot,\cdot)$ and $a(\cdot,\cdot)$ in~\eqref{def:bil} are not computable. We denote with a superscript $E$ the previously introduced forms restricted to the element $E$ and, { recalling property \textbf{(P1)}}, by $\V_k(E) = V_k(E) + \mathbb{P}_k(E)$. By following the standard
procedure for the Virtual Element Method \cite{BeBrCaMaMaRu2012}, we write
\begin{gather*}
    m^E(p, v) = (\rho \Pi^0_k p, \Pi^0_k v)_E + (\rho T^0_k p, T^0_k v)_E \approx m_k^E(p, v) =
    (\rho \Pi^0_k p, \Pi^0_k v)_{\ilario{E}} + \rho \ilario{h_E^2} s^E(T^0_k p, T^0_k v)_{\ilario{E}},
\end{gather*}
with $T^0_k = I - \Pi^0_k$.  The bilinear form $m_k^E: \V_k(E) \times \V_k(E)
\rightarrow \mathbb{R}$ is then an approximation of the local form $m^E$
composed by two computable parts:  the consistency and the stabilization term, respectively. 
\fd{The stabilization term can be any bilinear form that satifies specific properties, see, e.g.,~\cite{BeBrCaMaMaRu2012},
in this paper we use $s^E:\V_k(E) \times \V_k(E)
\rightarrow \mathbb{R}$ defined~as}
\begin{gather*}
    s^E(p, v) = \sum_{i=1}^{\sharp \rm dof} {\rm dof}_i(p) {\rm dof}_i(v)
    \quad \forall p, v \in \V_k(E),
\end{gather*}
where ${\rm dof}_i$ is the value of the $i$-th degree of freedom of the
argument and $\sharp {\rm dof}$ is the total number of degrees of freedom associated to $V_k(E)$.
Starting from the computability of $\Pi^0_k$ and $s_E$ { (cf. property \textbf{(P3)})} the bilinear form $m^E(\cdot,\cdot)$ is computable.

We follow the same approach for the form $a^E(\cdot,\cdot)$, by considering the decomposition
\begin{gather*}
    a^E(p, v) = (\mu \nabla p, \nabla v)_E = (\mu \bm{\Pi}_{k-1}^0 \nabla p,
    \bm{\Pi}_{k-1}^0 \nabla v)_E + (\mu \bm{T}_{k-1}^0 \nabla p,
    \bm{T}_{k-1}^0 \nabla v)_E,
\end{gather*}
 with $\bm{T}_{k-1}^0 = I - \bm{\Pi}_{k-1}^0$.
The form $a^E(\cdot,\cdot)$ is not computable since it contains virtual functions.
To have a computable form we introduce the stabilization form  $s^E$  before but scaled by a representative value of $\mu$ in $E$. Finally, we obtain
\ilario{
\begin{gather}\label{eq:a^E}
   a^E(p, v) \approx a_k^E(p, v) = (\mu \bm{\Pi}_{k-1}^0 \nabla p, \bm{\Pi}_{k-1}^0 \nabla v)_E
    +\mu \, s^E({T^\nabla_k} p, {T^\nabla_k} v)
\end{gather}
}
{where $T^\nabla_k = (I-\Pi^\nabla_k)$} and $a_k^E:\V_k(E) \times \V_k(E)\rightarrow \mathbb{R}$. 

%
%
%
%
%
%
%
Before stating the discrete weak form of Problem~\ref{eq:weak_form},
we note that:
\begin{itemize}
\item[-] Dirichlet boundary data are projected into the space $\widetilde{\mathbb{P}}_k$ and imposed in a strong way \textcolor{black}{(point-wisely)};
\item[-] the initial conditions are approximated by considering the \ilario{interpolation $(({p}_0)_I, ({p}_1)_I)$ of $(p_0, p_1)$};
\item[-] the global bilinear forms $m_k:\V_k(\Omega_h) \times \V_k(\Omega_h) \rightarrow\mathbb{R}$
and $a_k:\V_k(\Omega_h) \times \V_k(\Omega_h) \rightarrow\mathbb{R}$
are given by
\begin{gather*}
    m_k(p, v) = \sum_{E \in\Omega_h} m_k^E(p, v)
    \quad \text{and} \quad
    a_k(p, v) = \sum_{E \in\Omega_h} a_k^E(p, v)
    \quad \forall p, v \in \V_k(\Omega_h),
\end{gather*}
where $\V_k(\Omega_h) = V_k(\Omega_h) + \prod_{E\in\Omega_h} \mathbb{P}_k(E)$;
\item[-] the discrete functional $F_k(v): \V_k(\Omega_h)
\rightarrow\mathbb{R}$ is given from the local projection
\begin{gather*}
    F_k(v) = \sum_{E \in\Omega_h} (f, \ilario{\Pi_{k}^0}  v)_E
    \quad \forall v \in \V_k(\Omega_h);
\end{gather*}
\item[-] the bilinear form $c$ is computable because a function $v\in \V_k(\Omega_h)$ resctricted to $\Gamma_{A}$ is a mapped polynomial in $\widetilde{\mathbb{P}}_k(e)$.
\end{itemize}
 The discretized problem can be written as in the following.
\begin{problem}[Wave problem - Virtual Element formulation]\label{eq:semi_discr_weak_form}
    The Virtual Element formulation of Problem \ref{eq:weak_form} is:
    for any time $t\in (0,T]$ find $p_h = p(t) \in V_k(\Omega_h)$ such that
    \begin{gather*}%
        \begin{aligned}
            &m_k(\ddot{p}_h, v_h) + c(\dot{p}_h, v_h) + a_k(p_h, v_h) = F_k(v_h) && \forall v_h \in
            V_k(\Omega_h),\\
            &(p_h(0), \dot{p}_h(0)) = \ilario{(({p}_0)_I, ({p}_1)_I}).
        \end{aligned}
    \end{gather*}
\end{problem}

    \paragraph{Algebraic formulation.}

We start by introducing the matrices
associated to Problem \ref{eq:semi_discr_weak_form}. We indicate with $\varphi$ an element of the basis of $V_k(\Omega_h)$ and set
\begin{gather*}
    \begin{aligned}
        &M \in \mathbb{R}^{n\times n} && [M]_{ij} = m_k(\varphi_j, \varphi_i), &&
        &&
        C \in \mathbb{R}^{n\times n} &&[C]_{ij} = c(\varphi_j, \varphi_i),\\
        &A \in \mathbb{R}^{n\times n} && [A]_{ij} = a_k(\varphi_j, \varphi_i), &&
        &&
        \bm{f} \in \mathbb{R}^{n} && [\bm{f}]_{i} = F_k(\varphi_i),
    \end{aligned}
\end{gather*}
where $n={\rm dim} V_k(\Omega_h)$ and $[\cdot]_{ij}$ denotes the entry at row $i$ and column $j$ of the matrix
in the square brackets, likewise for a vector. Now, we can formulate the following problem.
\begin{problem}[Wave problem - Semi-discrete formulation]
    The semi-discrete fomulation of Problem \ref{eq:semi_discr_weak_form} is
    the following:
    for any time $t\in (0,T]$ find $\bp = \bp(t) \in \mathbb{R}^{n}$ such that
    \begin{gather*}
        \begin{aligned}
            &M \ddot{\bp}(t) + C \dot{\bp}(t) + A {\bp}(t) = \bfe(t)  &&  t \in (0,T],\\
            &(\bp(0), \dot{\bp}(0)) = (\bp_0, \bp_1), &&
        \end{aligned}.
    \end{gather*}
having set $\bp_0 = ({p}_0)_I$ and  $\bp_1 = ({p}_1)_I$.
\end{problem}

\toCheck{
\begin{remark}
In the following section we present the stability and convergence analysis for Problem \ref{eq:semi_discr_weak_form}
since we are more interested in the properties of spatial discretization. The analysis of the problem discretized both in space and time is beyond the scope of the paper. The latter can be obtained by combining the following results with classical convergence results for finite difference discretizations for Cauchy problems.
\end{remark}
}

\section{Theoretical analysis}\label{Sc:TeorhericalAnalysis}

{ In this section we prove the stability and the convergence of the semi-discrete virtual element approximation in the energy norm \eqref{eq:three-bar-norm}. The stability of the discrete solution is showed in Theorems~\ref{theorem:semi-discrete:stability}, whereas the a priori error estimates of the approximation error are derived in Theorem~\ref{theorem:semi-discrete:convergence}.
We start by recalling two important Lemmas.
}


\begin{lemma}\cite[Lemma 3.10]{BeiraoRussoVacca_2019}\label{lemma_cont}
Let $E \in \Omega_h$. Under the asssumptions \textbf{A1}--\textbf{A3}
and for any $\varepsilon \in (0,1/2)$ the following inequality holds
\begin{equation}\label{stability1_S}
s^E(p,p) \leq C h_E^{-2} \| p \|_{0,E}^2 + C h_E^{2\varepsilon} |p|_{1+\varepsilon,E}^2 \quad for \; all \; p \in \mathbb{V}_k(E),
\end{equation}
where the constant $C$ depends on $k, \varepsilon$, the shape regularity constant $\varrho$ and the map $\gamma$.
\end{lemma}
\begin{lemma}\cite[Lemma 3.12]{BeiraoRussoVacca_2019}\label{lemma_coer}
Let $E \in \Omega_h$. Under the asssumptions \textbf{A1}--\textbf{A3} the following inequality holds
\begin{equation}\label{stability2_S}
s^E(p,p) \geq C | p |_{1,E}^2 \quad for \; all \; p \in V_k(E),
\end{equation}
where the constant $C$ depends on $k$ and the shape regularity constant $\varrho$.
Moreover it holds
\begin{equation}\label{stability2_S_L2}
 h_E^{2} s^E(p,p) \geq C \| p \|_{0,E}^2 \quad for \; all \; p \in V_k(E).
\end{equation}
\end{lemma}

\noindent Next, we introduce the following results for the discrete bilinear forms $m_k^E(\cdot,\cdot)$ and $a_k^E(\cdot,\cdot)$.
\begin{proposition}\label{prop:stabm}
(k-consistency).For all $p \in V_k(E)$ and for all
  $q_k \in \mathbb{P}_k(E)$ it holds
  \begin{align}\label{eq:msh:k-consistency}
 m_k^E(p,q_k) = m^E(p,q_k)\,.
  \end{align}
(Stability). For any $\varepsilon \in (0,1/2)$ there exist two uniform positive constants
  $\mu_*,\,\mu^*$, such that for any element $E \in \Omega_h$ it holds that
  \begin{align}
    m_k^E(p,p) & \geq \mu_* \, m^E (p,p), \label{eq:msh:stability_1} \\
    m_k^E(p+q_k,p+q_k) & \leq \mu^* \left( \|p + q_k \|_{0,E}^2 + \|(I-\Pi_k^0) p\|_{0,E}^2 + h^{2\varepsilon+2}|(I-\Pi_k^0) p|_{1+\varepsilon,E}^2\right), \label{eq:msh:stability_2}
  \end{align}
  for all $p \in V_k(E)$ and $q_k \in \mathbb{P}_k(E)$.
\end{proposition}
\begin{proof}
Property \eqref{eq:msh:k-consistency} follows from the definition of the bilienar form $m_k^E(\cdot,\cdot)$. To prove inequality \eqref{eq:msh:stability_1}, recalling that $s^E(\Pi_k^0 p,\cdot)=0$,   we use Lemma~\ref{lemma_coer} and simple algebra to get
   \begin{align*}
    m^E (p,p) \lesssim h_E^2 s_E(p,p) \lesssim h_E^2 s_E(\Pi_k^0 p,\Pi_k^0 p) + h_E^2 s_E(T_k^0 p, T_k^0 p).
  \end{align*}
Then, by employing Lemma~\ref{lemma_cont} to \toCheck{the first term of the right-hand side} we have
 \begin{align*}
    m^E (p,p)   \lesssim \| \Pi_k^0 p \|^2_{0,E} + h_E^{2\varepsilon+2} | \Pi_k^0 p |^2_{1+\varepsilon,E} + h_E^2 s_E(T_k^0 p ,T_k^0 p).
  \end{align*}
Finally, a standard polynomial inverse estimate on star-shaped { domains} yields
 \begin{align*}
h_E^{2\varepsilon+2} | \Pi_k^0 p |_{1+\varepsilon,E}^2 \lesssim h_E^{2} | \Pi_k^0 p |^2_{1,E}  \lesssim \| \Pi_k^0 p \|^2_{0,E}.
  \end{align*}
The thesis follows by noting that $m^E_k(\Pi_k^0 p,\Pi_k^0 p) = \| \Pi_k^0 p \|^2_{0,E} + h_E^2 s_E(T_k^0 p, T_k^0 p)$.

Concerning inequality \eqref{eq:msh:stability_2}, we first observe that $(I-\Pi_k^0)q_k = 0$ for any $q_k\in\mathbb{P}_k(E)$, yielding 
\begin{align*}
    m_k^E(p+q_k,p+q_k) = \|\Pi_k^0(p + q_k) \|_{0,E}^2 + h_E^2 s_E (T_k^0 p, T_k^0 p).
\end{align*}
Next, by applying Lemma~\ref{lemma_cont} to the above inequality and using the continuity of $\Pi_k^0$ with respect to the $L^2$-norm we infer
\begin{align*}
    m_k^E(p+q_k,p+q_k) \lesssim \| p + q_k \|_{0,E}^2 + \| T_k^0 p\|_{0,E}^2 + h_E^{2\varepsilon+2} | T_k^0 p |^2_{1+\varepsilon,E}.
\end{align*}
and that concludes the proof.
\end{proof}

%
%
%
%
%
%
%
%
%
%
%
%
%
%
%
%
%
%
%
\begin{proposition}\label{prop:staba}
(k-consistency). For all $p \in V_k(E)$ and for all
  $q_k \in \mathbb{P}_k(E) $ it holds
  \begin{align}
    \label{eq:k-consistency}
    a_k^E(p,q_k) = a^E(p,q_k)\,.
  \end{align}
(Stability). For any $\varepsilon \in (0,1/2)$ there exists two uniform positive constants
  $\alpha_*,\,\alpha^*$, such that for any element $E \in \Omega_h$ it holds that
  \begin{align}
    a_k^E(p,p) & \geq \alpha_* \, a^E (p,p)     \label{eq:ash:stability_1}
 \\
    a_k^E(p+q_k,p+q_k) & \leq \alpha^* \left( |p + q_k |_{1,E}^2 + |(I-\Pi^\nabla_0) p|_{1,E}^2 + h^{2\varepsilon}|(I-\Pi^\nabla_0) p|_{1+\varepsilon,E}^2 \right),    \label{eq:ash:stability_2}
  \end{align}
  for all $p \in V_k(E)$ and $q \in \mathbb{P}_k(E)$.
\end{proposition}
\begin{proof}
Property \eqref{eq:k-consistency} follows from the definition of the bilinear form $a_k^E(\cdot,\cdot)$. The proofs of  inequalities  \eqref{eq:ash:stability_1}--\eqref{eq:ash:stability_2} can be obtained by following closely the proof of \cite[Proposition 3.13]{BeiraoRussoVacca_2019} and by noting that
\begin{align*}
\|\Pi^{k-1}_0 \nabla p \|_{0,E}^2 & = \| \nabla p \|_{0,E}^2  - \| (I-\Pi^{k-1}_0 ) \nabla p \|_{0,E}^2
\geq \| \nabla p \|_{0,E}^2  - \| \nabla (I-\Pi^{\nabla}_k) p \|_{0,E}^2
 = \| \nabla \Pi^\nabla_k p \|_{0,E}^2.
\end{align*}
\end{proof}

\subsection{Stability}
\label{subsec:stability}

{ We now address the stability analysis for the solution of Problem \ref{eq:semi_discr_weak_form}. 
First of all we
define the \emph{energy} norm}
\begin{align}
  \TNORM{p(t)}{h}^2
  =  m_k(\dot{p},\dot{p})(t) +
  a_k(p,p)(t)
  \qquad t\in[0,T],
  \label{eq:three-bar-norm}
\end{align}
which is defined for all $ p \in V_k(\Omega_h)$.
The local stability property of the bilinear forms $m_k(\cdot,\cdot)$
and $a_k(\cdot,\cdot)$ readily imply the  relation \begin{align}
\TNORM{p (t)}{}^2 = \| \rho^{1/2} \dot{p}(t) \|_0^2 + | \mu^{1/2} p(t) |_{1}^2 \lesssim  \TNORM{p(t)}{h}^2   \label{eq:energy:norm:equivalence}
\end{align}
for all time-dependent virtual element functions $p(t)$ with square
integrable derivative $\dot{p}(t)$.


\begin{theorem}\label{theorem:semi-discrete:stability}
  Let $f\in L^2( (0,T]; L^2(\Omega) )$ and let
  $p_h \in C^{2}((0,T]; V_k(\Omega_h))$ be the solution of Problem~\ref{eq:semi_discr_weak_form}.
  Then, it holds
  \begin{equation}
    \TNORM{p_h(t)}{h} \lesssim
    \TNORM{ (p_{0})_I }{h} + \int_0^t \NORM{ f(\tau) }{0,\Omega} d\tau.
  \end{equation}
\end{theorem}
\begin{proof}
  We substitute $v=\dot{p}_{h}(t)$
  in Problem~\ref{eq:semi_discr_weak_form} and, for all $t\in(0,T]$, we obtain
  \begin{align}
    m_k(\ddot{p}_{h},\dot{p}_{h})+ a_k(p_h,\dot{p}_{h}) +\ilario{ c(\dot{p}_h,\dot{p}_h)} =F_k(\dot{p}_{h}).
    \label{eq:LHS}
  \end{align}
  Since both $m_k(\cdot,\cdot)$ and $a_k(\cdot,\cdot)$ are symmetric
  bilinear forms, a straightforward calculation yields
  \begin{align*}
    \frac{1}{2}\frac{d}{dt}\big( m_k(\dot{p}_h,\dot{p}_h) + a_k(p_h,p_h) \big) =
    m_k(\ddot{p}_{h},\dot{p}_{h})+a_k(p_h,\dot{p}_h) .
  \end{align*}
  We substitute this expression in the left-hand side
  of~\eqref{eq:LHS}, we observe that $\ilario{ c(\dot{p}_h,\dot{p}_h)} \geq 0 $,  we integrate in time the resulting equation from
  $0$ to the intermediate time $t$, and using the definition of norm
  $\TNORM{\,\cdot\,}{h}$ in~\eqref{eq:three-bar-norm}, we find
  that
  \begin{align*}
    \TNORM{p_h(t)}{h}^2
    &\lesssim
    m_k(\dot{p}_{h}(t),\dot{p}_{h}(t)) + a_k(p_h(t),p_h(t))
    \\
    &= m_k(\dot{p}_{h}(0),\dot{p}_{h}(0)) + a_k(p_h(0),p_h(0))
     +  2\int_{0}^{t}F_k(\dot{p}_{h}(\tau))d\tau
    \\
    &\lesssim
    \TNORM{p_h(0)}{h}^2
    +\int_{0}^{t}F_k(\dot{p}_{h}(\tau))d\tau.
  \end{align*}
  Using that $(p_h(0),\dot{p}_h(0))=((p_{0})_{I},(p_{1})_{I})$, and  the  Cauchy-Schwarz inequality, we find that
  \begin{align*}
    \TNORM{p_h(t)}{h}^2
    \lesssim\TNORM{ (p_{0})_{I}}{h}^2 +  \int_0^t \NORM{f(\tau)}{0}\,\|\dot{p}_{h}(\tau)\|_{0}\,d\tau.
  \end{align*}
  The thesis follows on applying \eqref{eq:energy:norm:equivalence} and the Gronwall's Lemma~\cite[Lemma~A5,~p.~157]{Brezis:1973}.
\end{proof}

\subsection{Convergence analysis}

{
The aim of the present subsection is to show the convergence property of the proposed scheme.
We start our analysis recalling a classical approximation result for polynomials on star-shaped domains, see for instance \cite{brenner-scott:book}. 
}

\begin{lemma}
Let $E \in \Omega_h$, and let two real non-negative numbers $r,s$ with $r \leq s \leq k+1$. Then for all $p \in H^s(E)$, there exists a polynomial function $p_\pi \in \mathbb{P}_k(E)$ such that
\begin{align}
|p-p_\pi|_{r,E} & \leq C h_E^{s-r} |p|_{s,E},\label{eq:stima_pi}
\\
\|p-p_\pi\|_{0,E} & \leq C h_E^{s} \|p\|_{s,E},\label{eq:stima_pi_0}
\end{align}
with $C$ depending only on the polynomial degree $k$ and the shape regularity constant $\varrho$.
\end{lemma}
%
{
We now mention the following result concerning the optimal order of accuracy in $H^1$ or higher order norms for the virtual space $V_k$ (the proof follows combining Theorem 3.7 in \cite{BeiraoRussoVacca_2019} and Theorem 11 in \cite{apos}).

\begin{lemma}
Let any real number $\varepsilon \in [0,1/2)$ and
 $p \in H^s(\Omega) \cap V$, with $1+\varepsilon < \frac{3}{2} \leq s \leq k+1$. Then there exists a virtual element function $p_I \in V_k$ such that
\begin{align}
\sum_{E\in\Omega_h}|p-p_I|_{1+\varepsilon,E} & \leq C h_E^{s-1-\varepsilon} \|p\|_{s},\label{eq:stima_I}
\end{align}
with $C$ depending on the polynomial degree $k$, the shape regularity constant $\varrho$ and the parametrization $\gamma$.
\end{lemma}

We extend the result in the previous lemma by proving that the virtual space on curved elements $V_k$  has the optimal approximation order also in  $L^2$-norm.

\begin{lemma}
Let $p \in H^s(\Omega) \cap V$, with $\frac{3}{2} \leq s \leq k+1$. Then there exists a virtual element function $p_I \in V_k$ such that
\begin{align}
\sum_{E\in\Omega_h}\|p-p_I\|_{0,E} & \leq C h_E^{s} \|p\|_{s}\label{eq:stima_I_0},
\end{align}
with $C$ depending on the polynomial degree $k$, the shape regularity constant $\varrho$ and the parametrization $\gamma$.
\end{lemma}

\begin{proof}
We only sketch the proof since it follows the guidelines of Theorem 3.7 in \cite{BeiraoRussoVacca_2019}.
We preliminary observe that the function $p_I \in V_k$ that realizes \eqref{eq:stima_I} can be chosen in such a way $p(z) - p_I(z) = 0$ for each vertex/edge node $z$  (cf. the fist two items in \textbf{(P2)}).
Therefore employing Lemma 3.2 in \cite{BeiraoRussoVacca_2019} it holds that
\begin{equation}
\label{eq:inter-edge}
\|p - p_I\|_{0,e} \lesssim h_e^{s-1/2} \|p\|_{s-1/2, e} 
\qquad \text{for any mesh edge $e$.}
\end{equation}
For any $E\in \Omega_h$, from the Poincar\'e inequality we infer
\begin{align*}
\| p - p_I \|_{0,E} & \lesssim | \int_{\partial E} (p - p_I) \, ds | + h_E |p - p_I|_{1,E}  
\lesssim h_E^{\frac{1}{2}} \| p- p_I \|_{0,\partial E} + h_E |p - p_I |_{1,E} \,.
\end{align*}
Then the above equation and \eqref{eq:inter-edge} imply
\begin{align*}
\| p-p_I \|_{0,E} & \lesssim \sum_{e \in \partial E} h_e^{\frac{1}{2}} \| p-p_I \|_{0,e} + h_E | p-p_I |_{1,E} \\
& \lesssim \sum_{e \in \partial E}h_e^{s} \| p \|_{s-\frac{1}{2},e} + h_E | p-p_I |_{1,E} \\
& \lesssim h_e^{s} \| p \|_{s,E} + h_E | p-p_I |_{1,E},
\end{align*}
where in the last step we have use the trace inequality. The proof is completed by summing over all the mesh element $E$ and using  \eqref{eq:stima_I}. 
\end{proof}
}

{In the following for the sake of presentation we consider the case $\Gamma_A = \emptyset$, i.e., $c(\cdot, \cdot) = 0$. \toCheck{The general case  can be obtained in a similar way}.}

\begin{theorem}\label{theorem:semi-discrete:convergence}
  Let $s\in\mathbb{N}$ and $p\in C^2 \big( (0,T]; H^{s+1}(\Omega) \big)$,
  be the exact solution of problem~\eqref{eq:weak_form}.
  Let $p_h\in V_k$ be the solution of the semi-discrete
  problem~\eqref{eq:semi_discr_weak_form}.
  Then, under the mesh regularity
  assumptions of \textbf{A1}--\textbf{A3}, 
   for all $t\in[0,T]$, { all piecewise polynomials
  $p_{\pi}(t) \in \mathbb{P}_k(\Omega_h)$ and all interpolant functions $p_I(t) \in V_k$ approximating $p(t)$}, it holds 
  \begin{align}
    \TNORM{p(t)-p_h(t)}{}^2\lesssim
    \sup_{\tau\in[0,T]}\mathcal{H}_0^2(\tau)
    +\int_{0}^{t}\mathcal{H}_1^2(\tau)d\tau,
    \label{eq:semi-discrete:convergence}
  \end{align}
  where
  \begin{align}
    \mathcal{H}_0^2(\tau) &= \TNORM{{p}(\tau)-{p}_\pi(\tau)}{}^2 + m_k(\dot{p}_I(\tau)-\dot{p}_{\pi}(\tau),\dot{p}_I(\tau)-\dot{p}_{\pi}(\tau)) \nonumber  \\&\quad \quad+ a_k(p_I(\tau)-p_{\pi}(\tau),p_I(\tau)-p_{\pi}(\tau))
    \label{eq:semi-discrete:convergence:G0}
    \\
    \mathcal{H}_1^2(\tau) &=
 \TNORM{\dot{p}(\tau)- \dot{p}_\pi(\tau)}{}^2 + m_k(\ddot{p}_I(\tau)-\ddot{p}_{\pi}(\tau),\ddot{p}_I(\tau)-\ddot{p}_{\pi}(\tau)) \nonumber \\&\quad \quad+ a_k(\dot{p}_I(\tau)-\dot{p}_{\pi}(\tau),\dot{p}_I(\tau)-\dot{p}_{\pi}(\tau))  +
    \left(\sup_{p_h\in\mathbb{V}_k(\Omega_h) \backslash{\{0\}}} \frac{ | F(p_h) - F_k(p_h) | }{ |p_h|_1 }\right)^2 .
    \label{eq:semi-discrete:convergence:G1}
  \end{align}
\end{theorem}

\begin{proof}
We start by observing that from triangle inequality and  \eqref{eq:energy:norm:equivalence} it holds
  $$\TNORM{p(t)-p_h(t)}{}^2\leq\TNORM{p(t)-p_I(t)}{}^2+\TNORM{p_I(t)-p_h(t)}{h}^2.$$
 We bound the first term  by adding and subtracting $p_\pi$, using the definition of the energy norm \eqref{eq:three-bar-norm} and Propostions~\ref{prop:stabm}~and~\ref{prop:staba} as follows
 \begin{align}
 \TNORM{p(t)-p_I(t)}{}^2 & \leq \TNORM{p(t)-p_\pi(t)}{}^2 + \TNORM{p_\pi(t)-p_I(t)}{h}^2  \nonumber \\
 &\leq \TNORM{p(t)-p_\pi(t)}{}^2 +
 m_k(\dot{p}_I(t)-\dot{p}_{\pi}(t),\dot{p}_I(t)-\dot{p}_{\pi}(t)) 
 \nonumber  \\ & \quad \quad 
 + a_k(p_I(t)-p_{\pi}(t),p_I(t)-p_{\pi}(t)) \,. \label{eq:proof:energy:00}
 \end{align}
Next, we focus on the term $\TNORM{p_\pi(t)-p_I(t)}{h}$ and consider the following error equation
  \begin{align}
    m(\ddot{p}(t),v_h) - m_k(\ddot{p}_{h}(t),v_h)
    + a(p(t),v_h) - a(p_{h}(t),v_h)
    = F(v_h) - F_k(v_h),
    \label{eq:proof:energy:05}
  \end{align}
  which holds for all $v_h\in V_k$.
  Next, we rewrite this equation as
  $\TERM{T}{1}+\TERM{T}{2}=\TERM{T}{3}$, with the definitions:
  \begin{align*}
    \TERM{T}{1} &:= m(\ddot{p}(t),v_h) - m_k(\ddot{p}_{h}(t),v_h),\nonumber\\
    \TERM{T}{2} &:= a(p(t),v_h) - a_k(p_h(t),v_h),\nonumber\\
    \TERM{T}{3} &:= F(v_h) - F_k(v_h),
  \end{align*}
  and we dropped out the explicit dependence on $t$ to simplify the
  notation.
  We analyze each term separately.
  First, we rewrite $\TERM{T}{1}$ as
  \begin{align*}
    \TERM{T}{1}
    = m_k(\ddot{p}_I -\ddot{p}_{h},v_h)
    + m (\ddot{p}   -\ddot{p}_{\pi},v_h)
    - m_k(\ddot{p}_I -\ddot{p}_{\pi},v_h)
  \end{align*}
  by adding and subtracting $\ddot{p}_{I}$ and
  $\ddot{p}_{\pi}$ to the arguments of $m(\cdot,\cdot)$ and
  $m_k(\cdot,\cdot)$ and noting that, from Proposition~\ref{prop:stabm} we get
  $m(\ddot{p}_{\pi},v_h)=m_k(\ddot{p}_{\pi},v_h)$ for all
  $v_h\in\mathbb{V}_k(\Omega_h)$.
  We also rewrite $\TERM{T}{2}$ as
  \begin{align*}
    \TERM{T}{2}
    = a_k(p_I-p_h,v_h)
    + a(p -p_{\pi},v_h)
    - a_k(p_I-p_{\pi},v_h)
  \end{align*}
  by adding and subtracting $p_I$ and $p_{\pi}$ to the arguments of
  $a(\cdot,\cdot)$ and $a_k(\cdot,\cdot)$ and noting that, from Proposition~\ref{prop:staba} we have
  $a(p_{\pi},v_h) = a_k(p_{\pi},v_h)$ for all
  $v_h\in\mathbb{V}_k(\Omega)$.
  Let $e_h=p_I-p_h$.
  It holds that $e_h(0)=\dot{e}_{h}(0)=0$ since
  $p(0)=p_I(0)$ and
  $\dot{p}_{h}(0)=\dot{p}_{I}(0)$.
  Then, using the definition of $e_h$, we reconsider the error
  equation
  \begin{align}
    \TERM{T}{1} + \TERM{T}{2}
    &= m_k(\ddot{e}_{h},v_h) + a_k(e_h,v_h)
    +  m(\ddot{p} -\ddot{p}_{\pi},v_h)
    -  m_k(\ddot{p}_{I} -\ddot{p}_{\pi},v_h)\nonumber\\
    &+ a (p -p_{\pi},v_h)
    -  a_k(p_I-p_{\pi},v_h)
    = F(v_h) - F_k(v_h) = \TERM{T}{3}.
    \label{eq:proof:energy:10}
  \end{align}
  Assume that $v_h\neq0$ and consider the inequality:
  \begin{align}
   \TERM{T}{3} \leq  | F(v_h) - F_k(v_h) |
    \leq \|F-F_k\|_{\mathbb{V}_k^*(\Omega_h)}\,|v_h|_1,
    \label{eq:proof:energy:10:a}
  \end{align}
    being $\mathbb{V}_k^*(\Omega_h)$ the dual space of
    $\mathbb{V}_k^*(\Omega_h)$ and with
    \begin{align*}
        \|F-F_k\|_{\mathbb{V}_k^*(\Omega_h)} = \sup_{v_h\in\mathbb{V}_k(\Omega_h) \backslash{\{0\}}} \frac{ | F(v_h) - F_k(v_h)| }{ |v_h|_1 }.
    \end{align*}
  Note that there hold:
  \begin{align}
    m_k(\ddot{e}_h,\dot{e}_h) + a_k(e_h,\dot{e}_{h})=\frac{1}{2}\frac{d}{dt}\big( m_k(\dot{e}_{h},\dot{e}_{h}) + a_k(e_h,e_h) \big),
    \label{eq:proof:energy:15:a}\\
    | F(\dot{e}_h) - F_k(\dot{e}_h) |
    \leq  \|F-F_k\|_{\mathbb{V}_k^*(\Omega_h)}\,\TNORM{e_h}{}.
    \label{eq:proof:energy:15:c}
  \end{align}
  Setting $v_h=\dot{e}_{h}(t)$ on the left-hand side
  of~\eqref{eq:proof:energy:10} and
  employing~\eqref{eq:proof:energy:15:a}-\eqref{eq:proof:energy:15:c}
  together with~\eqref{eq:proof:energy:05}, we obtain, after
  rearranging the terms, that:
  \begin{align}
    &\frac{1}{2}\frac{d}{dt}\big( m_k(\dot{e}_{h},\dot{e}_{h}) + a_k(e_h,e_h) \big)
    \leq
    {}-m (\ddot{p}        -\ddot{p}_{\pi},\dot{e}_{h})
    + m_k(\ddot{p}_{I} -\ddot{p}_{\pi},\dot{e}_{h})
    \nonumber\\
    &\qquad
    - a (p -p_{\pi},\dot{e}_{h})
    + a_k(p_I-p_{\pi},\dot{e}_{h})
    +  \|F-F_k\|_{\mathbb{V}_k^*(\Omega_h)}\,\TNORM{e_h}{}.
    \label{eq:proof:energy:25}
  \end{align}

  To ease the notation, we collect together the last two terms above
  and denote them by $\TERM{R}{1}(t)$ (note that they still depend on
  $t$).
  We integrate in time from $0$ to $t$ both sides
  of~\eqref{eq:proof:energy:25} and note that the initial term is zero
  since $e_h(0)=\dot{e_h}(0)=0$ getting
  \begin{align}
    \TNORM{e_h(t)}{h}^2
    &\leq m_k\big(\dot{e}_{h}(t),\dot{e}_{h}(t)) + a_k(e_h(t),e_h(t)\big) \big)
    \nonumber\\
    &\leq \int_{0}^{t}\Big(
    \TERM{R}{1}(\tau)
    - m \big(\ddot{p}(\tau)        -\ddot{p}_{\pi}(\tau),\dot{e}_{h}(\tau)\big)
    + m_k\big(\ddot{p}_{I}(\tau) -\ddot{p}_{\pi}(\tau),\dot{e}_{h}(\tau)\big)
    \nonumber\\
    &\phantom{ \leq \int_{0}^{t}\Big( }
    - a \big(p (\tau)-p_{\pi}(\tau),\dot{e}_{h}(\tau)\big)
    + a_k\big(p_I(\tau)-p_{\pi}(\tau),\dot{e}_{h}(\tau)\big)
    \Big)\,d\tau.
    \label{eq:proof:energy:30}
  \end{align}
  Then, we integrate by parts the integral that contains
  $a(\cdot,\cdot)$ and $a_k(\cdot,\cdot)$, and again use the fact
  that $e(0)=\dot{e}_{h}(0)=0$, to obtain
  \begin{align}
    \TNORM{e(t)}{h}^2
    &\leq \int_{0}^{t}\Big(
    \TERM{R}{1}(\tau)
    + \Big[
    {}-m \big(\ddot{p}(\tau)        -\ddot{p}_{\pi}(\tau),\dot{e}_{h}(\tau) \big)
    + m_k\big(\ddot{p}_{I}(\tau) -\ddot{p}_{\pi}(\tau),\dot{e}_{h}(\tau) \big)\Big]
    \nonumber\\
    &\phantom{ \leq \int_{0}^{t}\Big( }
    + \Big[
    a \big( \dot{p}(\tau)       -\dot{p}_{\pi}(\tau),  e_h(\tau) \big)
    -a_k\big( \dot{p}_{I}(\tau)-\dot{p}_{\pi}(\tau),e_h(\tau) \big) \Big]
    \Big)\,d\tau
    \nonumber\\
    &\phantom{ \leq \int_{0}^{t}\Big( }
    + \Big[
    {}-a \big( p(t) -p_{\pi}(t), e_h(t) \big)
    + a_k\big( p_I(t)-p_{\pi}(t), e_h(t) \big) \Big]
    \nonumber\\
    &= \int_{0}^{t}\Big(
    {}\TERM{R}{1}(\tau)
    + \TERM{R}{2}(\tau)
    + \TERM{R}{3}(\tau)
    \Big)\,d\tau
    + \TERM{R}{4}(t),
    \label{eq:proof:energy:35}
  \end{align}
  where terms $\TERM{R}{\ell}$, $\ell=2,3,4$, match with the squared
  parenthesis.
  We bound term $\TERM{R}{1}$ by using the Young's inequality and \eqref{eq:energy:norm:equivalence}
    \begin{align}
     {| \TERM{R}{1} |
     \leq  C
    \|F-F_k\|_{\mathbb{V}_k^*(\Omega_h)}^2  + \frac{1}{2}\TNORM{e_h}{h}^2.}
    \label{eq:proof:energy:36}
  \end{align}
  To bound $\TERM{R}{2}$ we use the continuity of $m(\cdot,\cdot)$, the tringle inequality for $m_k(\cdot,\cdot)$ and Young's inequality:
  \begin{align*}
    | \TERM{R}{2} |
    & \leq
    | m (\ddot{p}        -\ddot{p}_{\pi}, \dot{e}_{h} ) | +
    | m_k(\ddot{p}_{I} -\ddot{p}_{\pi}, \dot{e}_{h} ) | \nonumber \\
    & \leq  \frac{1}{2}\NORM{\ddot{p}-\ddot{p}_{\pi}}{0}^2 + \frac{1}{2}\NORM{ \dot{e}_{h} }{0}^2 + \frac{1}{2} m_k(\ddot{p}_{I}-\ddot{p}_{\pi}, \ddot{p}_{I}-\ddot{p}_{\pi}) + \frac{1}{2}m_k(\dot{e}_h,\dot{e}_h)
  \end{align*}
Finally, using \eqref{eq:msh:stability_1} to bound $\NORM{ \dot{e}_{h} }{0}^2$  one can easily get
  \begin{align}
    | \TERM{R}{2} |
    & \lesssim \NORM{\ddot{p}-\ddot{p}_{\pi}}{0}^2 + m_k(\ddot{p}_{I}-\ddot{p}_{\pi}, \ddot{p}_{I}-\ddot{p}_{\pi})  + m_k(\dot{e}_h,\dot{e}_h)
    \label{eq:proof:energy:40}
  \end{align}

  Similarly, to bound $\TERM{R}{3}$ we use the continuity of
  $a(\cdot,\cdot)$, the tringle inequality for $a_k(\cdot,\cdot)$,
  Young's inequality and  \eqref{eq:ash:stability_1} to bound  term $|\dot{e}_h|^2_1$ to get
  \begin{align}
    | \TERM{R}{3} |
    & \leq
    | a (\dot{p}        -\dot{p}_{\pi}, e_h ) | +
    | a_k(\dot{p}_{I} -\dot{p}_{\pi}, e_h ) | \nonumber\\
    & \leq \frac{1}{2}|\dot{p} -\dot{p}_{\pi}|_{1}^2 + \frac{1}{2} | e_h |^2_1
+ \frac{1}{2} a_k(\dot{p}_{I}-\dot{p}_{\pi}, \dot{p}_{I}-\dot{p}_{\pi}) + \frac{1}{2}a_k({e}_h,{e}_h) \nonumber \\
    & \lesssim |\dot{p} -\dot{p}_{\pi}|_{1}^2
+  a_k(\dot{p}_{I}-\dot{p}_{\pi}, \dot{p}_{I}-\dot{p}_{\pi}) + a_k({e}_h,{e}_h).
    \label{eq:proof:energy:45}
  \end{align}
  By prooceeding in the same way for $\TERM{R}{4}$ yields to
  \begin{align}
    | \TERM{R}{4} |
    & \leq
    {}| a ( p- p_{\pi}, e_h ) |
    + | a_k( p_I-p_{\pi}, e_h ) | \nonumber\\
   & \lesssim \frac{1}{2\delta}|{p} -{p}_{\pi}|_{1}^2 + \frac{\delta}{2}a_k({e}_h,{e}_h)
   +  a_k(p_{I}-p_{\pi}, p_{I}-p_{\pi}),
    \label{eq:proof:energy:50}
  \end{align}
  for $\delta >0$.
  Using bounds~\eqref{eq:proof:energy:36},~\eqref{eq:proof:energy:40}, \eqref{eq:proof:energy:45},
  and~\eqref{eq:proof:energy:50} in \eqref{eq:proof:energy:35}, we
  find the inequality
  \begin{align*}
    \TNORM{e_h(t)}{h}^2
    \lesssim \widetilde{\mathcal{H}}_0^2(t)
    +\int_{0}^{t}\mathcal{H}_1^2(\tau)\,d\tau,
    +\int_{0}^{t} \TNORM{e_h(\tau)}{h}^2\,d\tau,
  \end{align*}
  where $\widetilde{\mathcal{H}}_0^2(t) = |{p} -{p}_{\pi}|_{1}^2 +  a_k(p_{I}-p_{\pi}, p_{I}-p_{\pi})$ while
 $\mathcal{H}_1(t)$ is reported in \eqref{eq:semi-discrete:convergence:G1}.
  Again, an application of the Gronwall's Lemma~\cite[Lemma~A5,~p.~157]{Brezis:1973}
  yields
  \begin{align*}
    \TNORM{e_h(t)}{h}^2\lesssim \widetilde{\mathcal{H}}_0^2(\tau)+\int_{0}^{t}\mathcal{H}_1^2(\tau)d\tau,
  \end{align*}
  that combined with \eqref{eq:proof:energy:00} concludes the proof.
\end{proof}

\begin{corollary}
Under the assumptions of Theorem~\ref{theorem:semi-discrete:convergence}, for $f\in L^2((0,T]; H^{s-1}(\Omega))$ we have  that
\begin{align}
\sup_{t\in (0,T]} \TNORM{p-p_h}{}^2 & \lesssim h^{2s-2} \left( h^2 \| \dot{p} \|_s^2 + \| p \|_s^2  +  \int_0^T  h^2 \| \ddot{p}(\tau) \|_s^2 + \| \dot{p}(\tau) \|_s^2   + h^{2} |f(\tau)|^2_{s-1}  \, d\tau \right),
\end{align}
where $\frac{3}{2}\leq s \leq k+1$ and the hidden constant may depend on the model parameters and approxiamtion constants, the polynomial degree and the final observation time $T$.
\end{corollary}
\begin{proof}
The proof follows by estimating the terms $\mathcal{H}^2_0$ and $\mathcal{H}^2_1$ in \eqref{eq:semi-discrete:convergence:G0}
and \eqref{eq:semi-discrete:convergence:G1}, respectively.
We start by considering the term $\TNORM{p-p_\pi}{}^2$.
By applying \eqref{eq:stima_pi}--\eqref{eq:stima_pi_0}, it holds
\begin{equation}\label{eq:stima_pi_trebarre}
\TNORM{p-p_\pi}{}^2 \lesssim h^{2s-2} \left( h^2 \| \dot{p} \|^2_s + \| p \|^2_s \right).
\end{equation}
We now estimate the term $m_k(\dot{p}_{I}-\dot{p} _{\pi}, \dot{p}_{I}-\dot{p}_{\pi})$ by using  the result of Proposition~\ref{prop:stabm}. In particular, for all $E\in \Omega_h$ we obtain
\begin{align}
m_k^E(\dot{p}_{I}- \dot{p}_{\pi}, \dot{p}_{I} - \dot{p}_{\pi})&\lesssim \| \dot{p}_{I}-\dot{p}_{\pi}\|_{0,E}^2 + \| (I-\Pi_k^0)\dot{p}_{I}\|_{0,E}^2 + h^{2\varepsilon+2} | (I-\Pi_k^0)\dot{p}_{I}|_{1+\varepsilon,E}^2\nonumber \\
& = :  \TERM{T}{1} + \TERM{T}{2} + \TERM{T}{3}. \nonumber
\end{align}
The first term, using triangle inequality together with~\eqref{eq:stima_I_0}, is estimated as follows
\begin{align*}
\TERM{T}{1} & \leq \| \dot{p} - \dot{p}_I\|_{0,E}^2 + \| \dot{p} - \dot{p}_\pi\|_{0,E}^2
\lesssim h^{2s} \| \dot{p} \|^2_{s},
\end{align*}
Concerning the second term, by the continuity of the $\Pi^0_k$ projection and by using \eqref{eq:stima_I_0} we have
\begin{align*}
\TERM{T}{2} & \leq \| (I-\Pi_k^0)(\dot{p}-\dot{p}_I) \|_{0,E}^2 + \| (I-\Pi_k^0)\dot{p} \|_{0,E}^2 \leq \| \dot{p}-\dot{p}_I \|_{0,E}^2 + \| \dot{p} \|_{0,E}^2 \lesssim h^{2s} \| \dot{p} \|^2_{s},
\end{align*}
Finally the last term is handled using equation~\eqref{eq:stima_I} and standard polynomial inverse estimates on star-shaped
domains getting
\begin{align*}
\TERM{T}{3}  & \lesssim h^{2\varepsilon+2} | \dot{p}_{I} - \dot{p} |_{1+\varepsilon,E}^2
 + h^{2\varepsilon+2} | (I-\Pi_k^0)\dot{p} |_{1+\varepsilon,E}^2
+ h^{2\varepsilon+2} | \Pi_k^0(\dot{p}-\dot{p}_I) |_{1+\varepsilon,E}^2 \\
& \lesssim h^{2\varepsilon+2+2s-2\varepsilon-2} \| \dot{p} \|_{s}^2 + h^{2s}\| \dot{p} \|_{s}^2 + \| \dot{p}-\dot{p}_I\|_{0,E}^2 \lesssim h^{2s}\| \dot{p} \|_{s}^2
\end{align*}
Collecting all the estimates and summing over all the elements we obtain
\begin{equation}\label{eq:estimate_m}
m_k(\dot{p}_{I}- \dot{p}_{\pi}, \dot{p}_{I} - \dot{p}_{\pi}) \lesssim h^{2s}\| \dot{p} \|_{s}^2.
\end{equation}
By proceeding similarly we can obtain a bound for $a_k^E(p_{I}- p_{\pi}, p_{I} - p_{\pi})$ as follows
\begin{align}
a_k^E(p_{I}- p_{\pi}, p_{I} - p_{\pi})&\lesssim  | p_{I}-p_{\pi} |_{1,E}^2 + \| (I-\Pi^\nabla_0)p_{I}\|_{0,E}^2 + h^{2\varepsilon+2} | (I-\Pi^\nabla_0)p_{I}|_{1+\varepsilon,E}^2\nonumber \\
& \lesssim h^{2s-2}\| p \|_{s}^2. \label{eq:estimate_a}
\end{align}
Now, summing up \eqref{eq:stima_pi_trebarre}, \eqref{eq:estimate_m} and \eqref{eq:estimate_a} we have
\begin{equation}
\mathcal{H}^2_0 \lesssim h^{2s-2} \left( h^2 \| \dot{p} \|^2_s + \| p \|^2_s \right).
\end{equation}
Concerning the term $\mathcal{H}^2_1$ we
note that $$F_k(v_h) = \sum_{E\in\Omega_h} ( f ,\Pi_k^0 v_h )_E =  \sum_{E\in\Omega_h} ( \Pi_k^0 f  , v_h )_E$$ for any $v_h \in V_k$.
Then, it holds
\begin{align*}
| F(v_h) - F_k(v_h) | & \leq \sum_{E\in\Omega_h} |((I- \Pi_k^0) f , v_h )_E |  = \sum_{E\in\Omega_h} |(I- \Pi_k^0) f , (I-\Pi^0_0) v_h )_E | \\
& \lesssim h^{s-1} |f|_s \;h |v_h |_1.
\end{align*}
Consequently,
\begin{align*}
 \|F-F_k\|_{\mathbb{V}_k^*(\Omega_h)}^2 \lesssim h^{2s} |f|_{s-1}^2,
\end{align*}
Collecting the above inequality together with estimates \eqref{eq:estimate_m} and \eqref{eq:estimate_a} for the terms $m_k(\ddot{p}_{I}- \ddot{p}_{\pi}, \ddot{p}_{I} - \ddot{p}_{\pi})$
and  $a_k^E(p_{I}- p_{\pi}, p_{I} - p_{\pi})$, respectively,  we conclude the proof by observing that

\begin{equation*}
\mathcal{H}^2_1 \lesssim h^{2s-2}\left( h^2 \| \ddot{p} \|_s^2 + \| \dot{p} \|_s^2  + h^{2} |f|^2_{s-1} \right).
\end{equation*}

\end{proof}

\section{Numerical results}\label{Sc:NumericalResults}

In this section we consider three different test cases to verify \as{the theoretical results} and show the capabilities of the numerical scheme presented in this work. In particular, we consider geometries with curved interfaces at the boundary, or internal of the domain. We compare our approach with a classical Virtual Element discretization, for which the
geometry is not respected, \as{i.e., is approximated by straight edges}. We will show that, for the latter, the geometrical error will dominate the numerical error, leading to a loss of convergence order.

The aim of the first case presented in Section~\ref{subsec:numerical_example_error_decay} is to show the error decay when the analytical solution is known.
In the second case, discussed in Section~\ref{subsec:numerical_example_plane_wave}, we present the propagation of a plane wave in a heterogeneous domain with a circular inclusion. 
Finally, in Section~\ref{subsec:numerical_example_listric_fault}, an \as{idealization of a realistic} curved geometry is taken into account showing the
applicability of the method in a domain with complex interfaces.

In the following we name \texttt{withGeo} the current method and  \texttt{noGeo} the approach where curved edges are approximated by straight lines.

The problem so far considered is in a semi-discrete version, to derive the
fully-discrete problem we sub-divide the time interval $(0,T]$ into $N_T$ intervals with equal size $\Delta t$, we write $\bp^{(i)} = \bp(t_i)$ with $t_i = i \Delta
t$, for $i=0,..,N_T$.
To avoid limitation on the time step, we consider the following second
order implicit scheme
\begin{gather*}
\begin{aligned}
    \left[ M  + \frac{\Delta t}{2} C + \Delta t^2 A \right]\bp^{(i+2)}
     & = 2 M \bp^{(i+1)}
    + \left[\frac{\Delta t}{2} C - M \right] \bp^{(i)}+  \Delta t^2
    \bfe^{(i+2)}, \quad i \geq 0,\\
    M  \bp^{(1)} & = \left[ M - \frac{\Delta t}{2}A \right]\bp_0 - \Delta t(M + C) \bp_1 + \frac{\Delta t^2}{2}\bfe^{(0)}.
\end{aligned}
\end{gather*}

\subsection{\toCheck{Verification test}}\label{subsec:numerical_example_error_decay}

In this test case, we verify the error decay of the virtual element solution with respect to the mesh size $h$. 
\fd{We compare the results of the proposed method, \texttt{withGeo}, with respect 
to the ones obtained approximating the curved boundary with straight edges, \texttt{noGeo}, and
we see the impact of handling exactly the geometry on the numerical solution.}

 We consider Problem~\eqref{modelpb} posed in a circular ring having internal and external radii equal to $r_i = 0.5$ and $r_o = 1$, respectively.  Figure \ref{fig:ex1_grid} shows the computational domain.
\begin{figure}[tbp]
    \centering
    \includegraphics[width=0.33\textwidth]{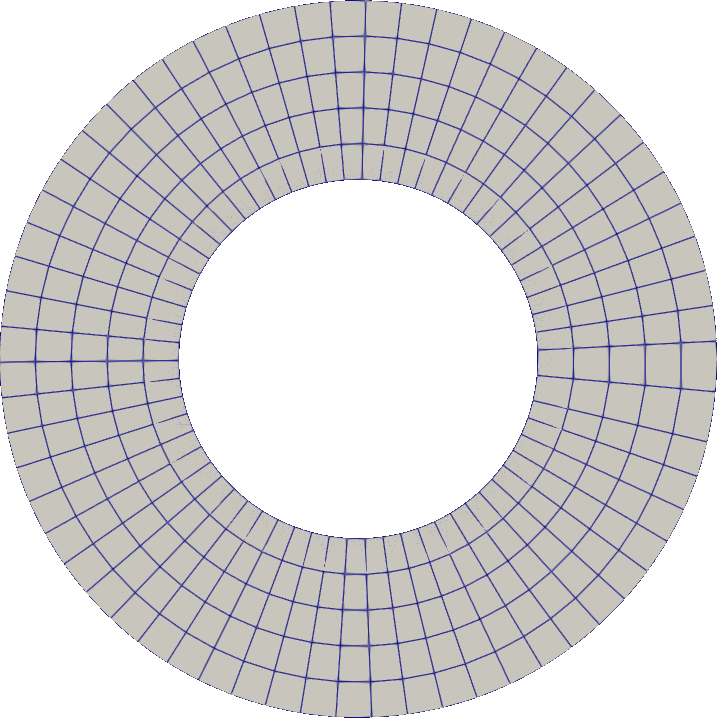}%
    \hspace*{0.1\textwidth}
    \includegraphics[width=0.33\textwidth]{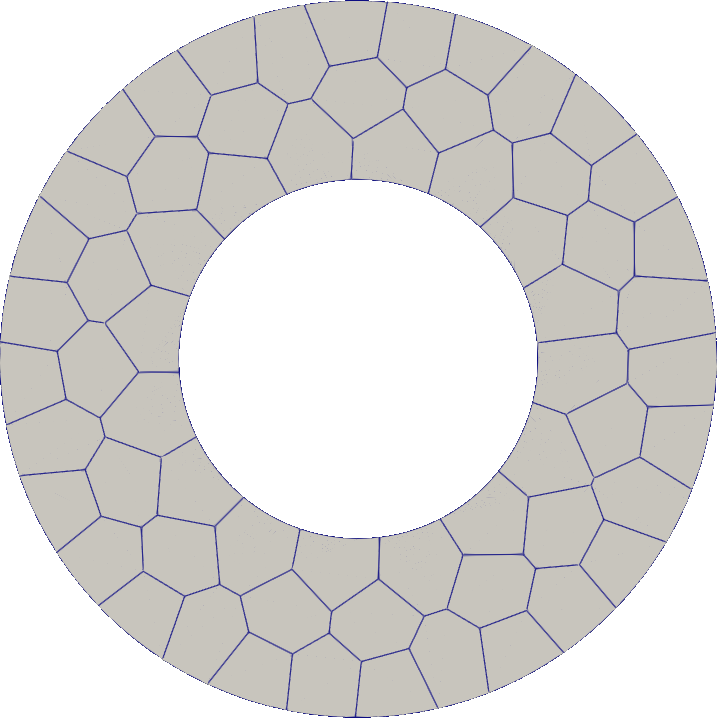}
    \caption{Computational domain for the example in section
    \ref{subsec:numerical_example_error_decay}. On the left a representative of the first family \texttt{quad} of grids, made of radial rectangles. On the right a grid from the second family, composed by polygons and named \texttt{poly}.}%
    \label{fig:ex1_grid}
\end{figure}
 We consider the following analytical solution
\begin{gather}\label{pex}
    p_{\rm ex}(\bm{x}, t) = \sin(0.5 \pi
    t) \sin(0.5 \pi {x}_1 {x}_2) (\Vert \bm{x} \Vert^2-r_i^2)(\Vert \bm{x} \Vert^2-r_o^2),
\end{gather}
together with  Dirichlet conditions on the boundary and
$\mu = 1$, $\rho = 1$. Source term and initial condition are computed accordingly to \eqref{pex}. 
We remark that solution is regular in space and time \toCheck{so that the hypothesis of Theorem~\ref{theorem:semi-discrete:convergence} are verfied}.

To \toCheck{consider only the space discretization error}, we set $\Delta t = 10^{-8}$ and compute the $H^1$ and $L^2$ errors in space at the end of the simulation time \toCheck{fiexd as} $T = 10 \Delta t$. \toCheck{
\begin{remark}
We notice that the asympotic behavoiur of the $H^1$-error can directly be inferred from the results of Theorem~\ref{theorem:semi-discrete:convergence}, while the $L^2$ error decay con be obtained by using similar arguments of those presented in \cite{AntoniettiManziniMazzieriMouradVerani_2019}. The latter is beyond the scope of this work.
\end{remark}}

We consider two families of meshes: the first named \texttt{quad} is composed by radial rectangles while the second \texttt{poly} is constructed from a Voronoi tessellation,  cf. Figure~\ref{fig:ex1_grid}.
In both cases the elements have curved edges at the boundary and straight internally. The errors in $L^2$ and
semi-norm $H^1$ are computed as
\begin{gather*}
    L^2 \text{ error} = \frac{\sqrt{\sum_{E \in \Omega_h}\Vert p_{\rm ex} -
    \Pi_k^0 p \Vert_{E}^2}}{\Vert p_{\rm ex} \Vert_{\Omega_h}}
    \quad \text{and} \quad
    H^1 \text{ error} = \frac{\sqrt{\sum_{E \in \Omega_h}\Vert \nabla p_{\rm ex} -
    \bm{\Pi}_{k-1}^0 \nabla p \Vert_{E}^2}}{\Vert \nabla p_{\rm ex} \Vert_{\Omega_h}}.
\end{gather*}
\begin{figure}[tbp]
    \centering
    \includegraphics[width=1\textwidth]{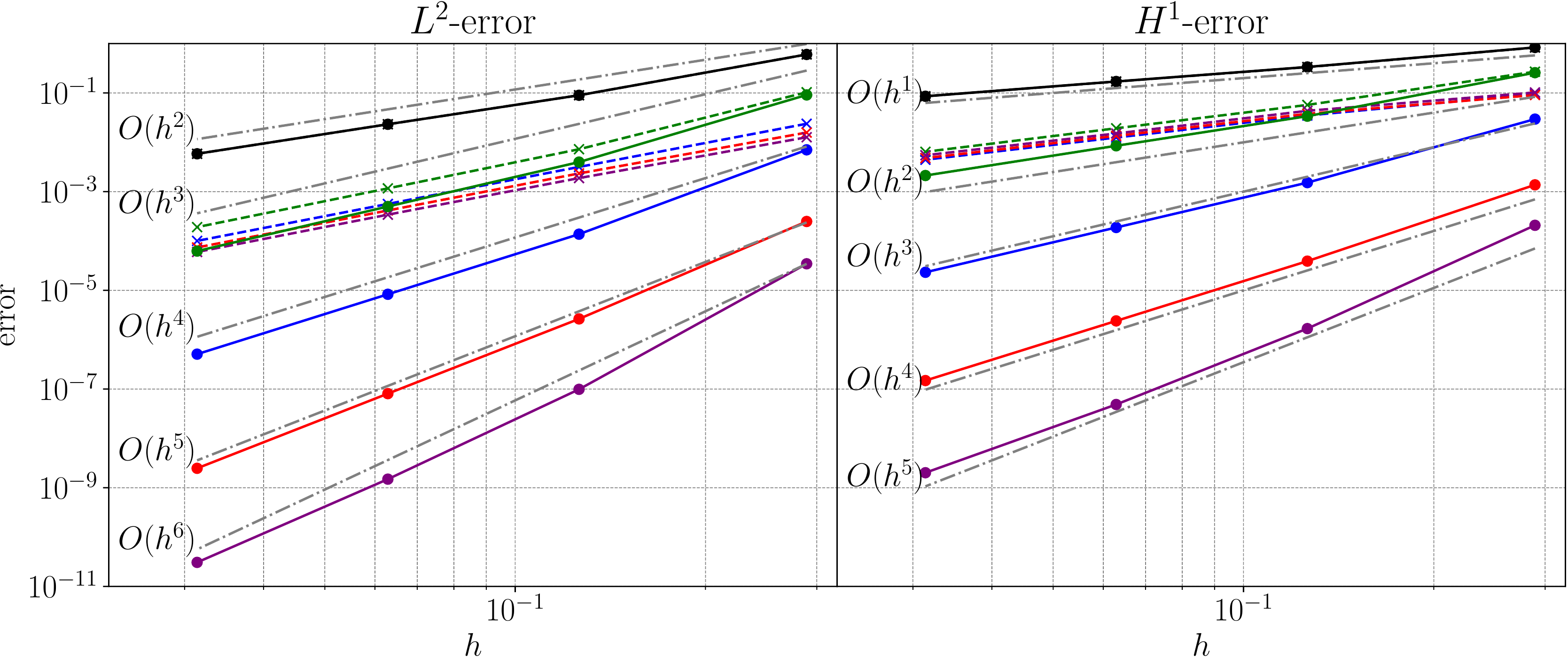}\\
    \includegraphics[width=1\textwidth]{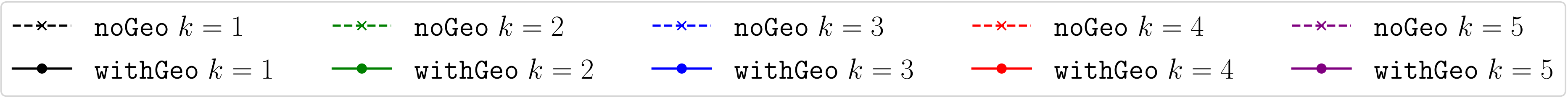}%
    \caption{$L^2$ and $H^1$ error decay for the family \texttt{quad}
    in the example in Subsection
    \ref{subsec:numerical_example_error_decay}.}%
    \label{fig:ex2_quad_error}
\end{figure}
\begin{figure}[tbp]
    \centering
    \includegraphics[width=1\textwidth]{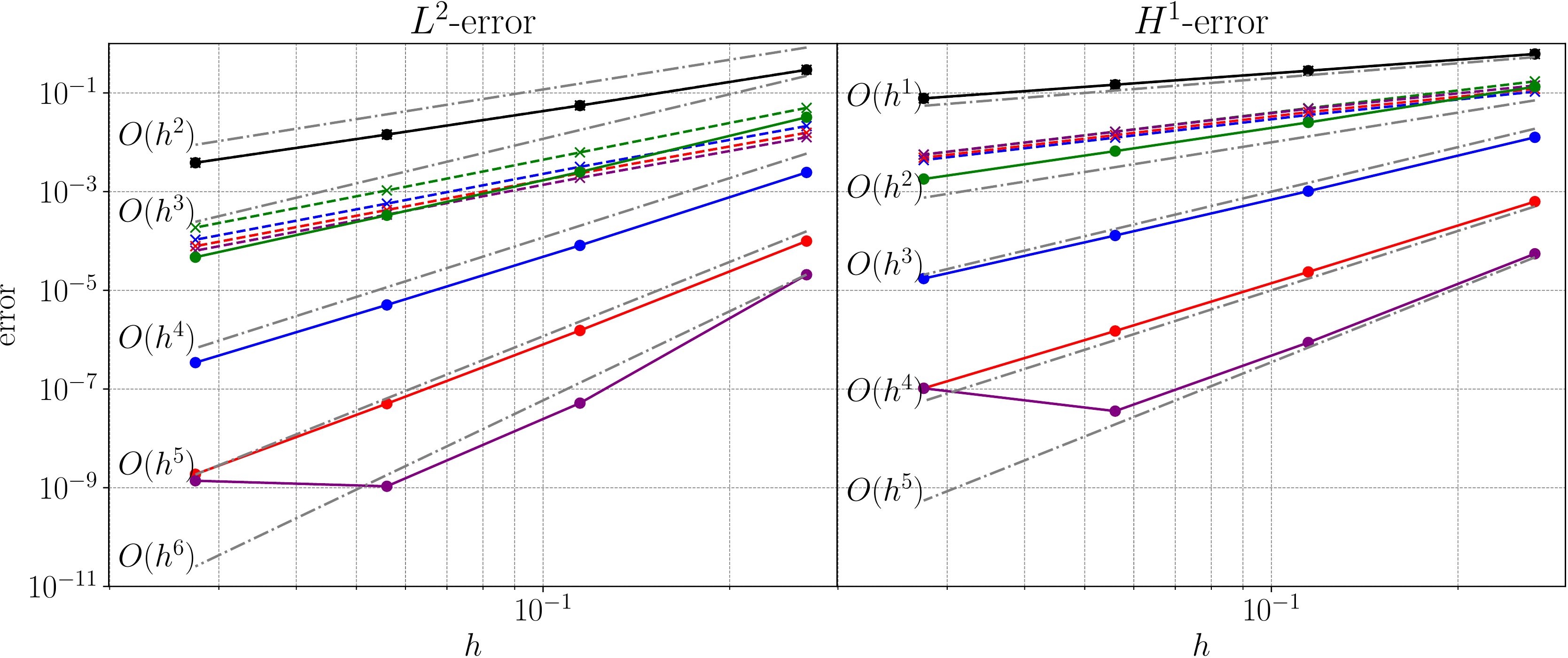}\\
    \includegraphics[width=1\textwidth]{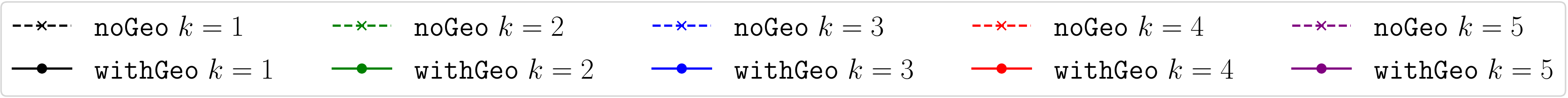}%
    \caption{$L^2$ and $H^1$ error decay for the family \texttt{poly}
    in the example in Subsection
    \ref{subsec:numerical_example_error_decay}.}%
    \label{fig:ex2_poly_error}
\end{figure}
Figure \ref{fig:ex2_quad_error} shows the error decay for the family \texttt{quad} and Figure
\ref{fig:ex2_poly_error} for the family \texttt{poly}.
In both cases we notice that for {$k = 1$}
the $H^1$ error for \texttt{withGeo} and \texttt{noGeo} decays as expected, i.e., as
$O(h^k)$. However, for {$k \geq 2$} the geometrical error dominates in the
\texttt{noGeo} and limits the error decay to $O(h^{3/2})$. For the \texttt{withGeo}
method the error decay behaves as expected reaching the convergence rate equal
to $O(h^k)$.

For the $L^2$-error the situation is similar, as we expect an error of convergence equal to $O(h^{k+1})$.
{ As before, the geometrical error \as{limits convergence rate}  for $k\geq 2$ for
\texttt{noGeo}.} Again, for the method \texttt{withGeo} the errors decay as
expected. For degree $k=5$ in the \texttt{poly} we notice a stagnation of the
error for small $h$  in both $L^2$ and $H^1$ norm, 
probably related to \toCheck{the numerical linear algebra}. 
\fd{A numerical proof of this inference is that such phenomena is not present in the \texttt{quad} family. 
Indeed, quadrilateral meshes have more regular shapes with respect to the polygonal ones and, consequenly,
the linear system arising from such discretization will have a lower condition number.}

We can conclude that, at least for this example, the proposed method is
an attractive approach to solve the wave equation with high order approximation in presence of curved boundaries.

\subsection{Plane wave test case}\label{subsec:numerical_example_plane_wave}

In this second case, a plane wave enters the computational domain $\Omega=(-1, 1)^2$ from the left boundary of  and encounters a circular inclusion with
different mechanical properties, i.e., a different $\mu$.
The inclusion is a circle of radius 0.2, i.e.,
$$
\gamma=\{\bm{x}: \Vert{\bm{x}} \Vert^2 - 0.2^2 = 0 \}\,.
$$
The computational domain is depicted in Figure \ref{fig:ex3_grid}.
The computational grid is constructed starting from
a Cartesian grid and then cutting all the elements that are crossed by $\gamma$.
The obtained grid is not extremely refined around $\gamma$
since the exact geometry is captured by curved edges,
indeed the number of elements is $6625$ compared to $6561$ of the original
Cartesian grid.
This fact represents an advantage from the computational point of view
since we do not have to increase the number of degrees of freedom
to capture the internal curved interface.
\begin{figure}[tbp]
    \centering
    \includegraphics[width=0.3\textwidth]{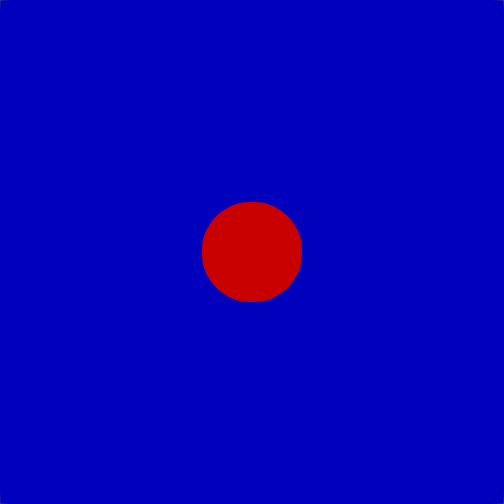}%
    \hspace*{0.05\textwidth}%
    \includegraphics[width=0.3\textwidth]{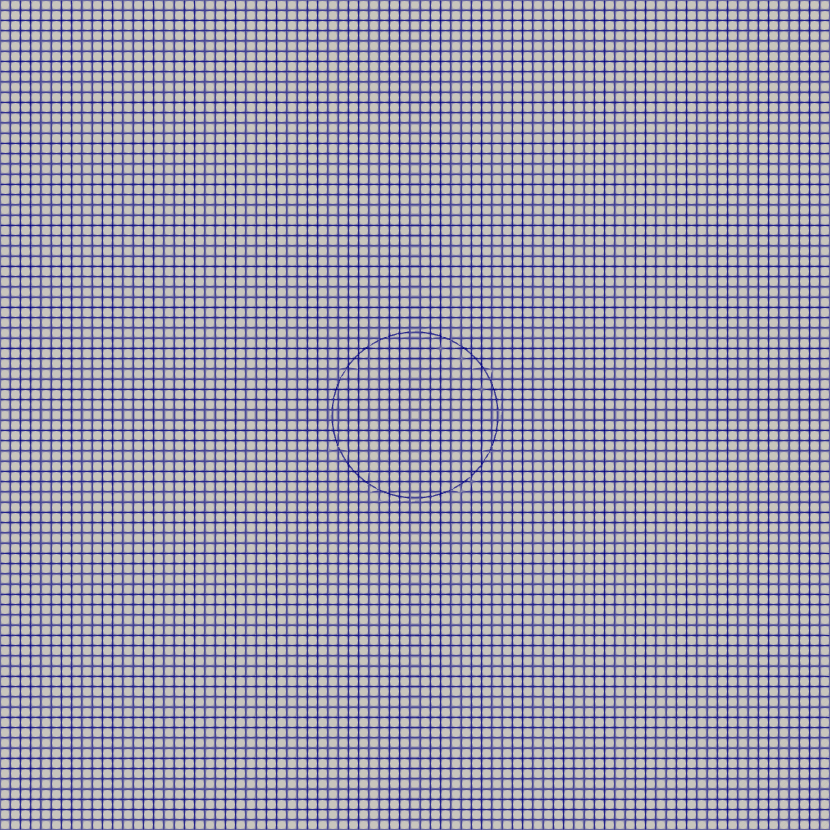}%
    \hspace*{0.05\textwidth}%
    \includegraphics[width=0.3\textwidth]{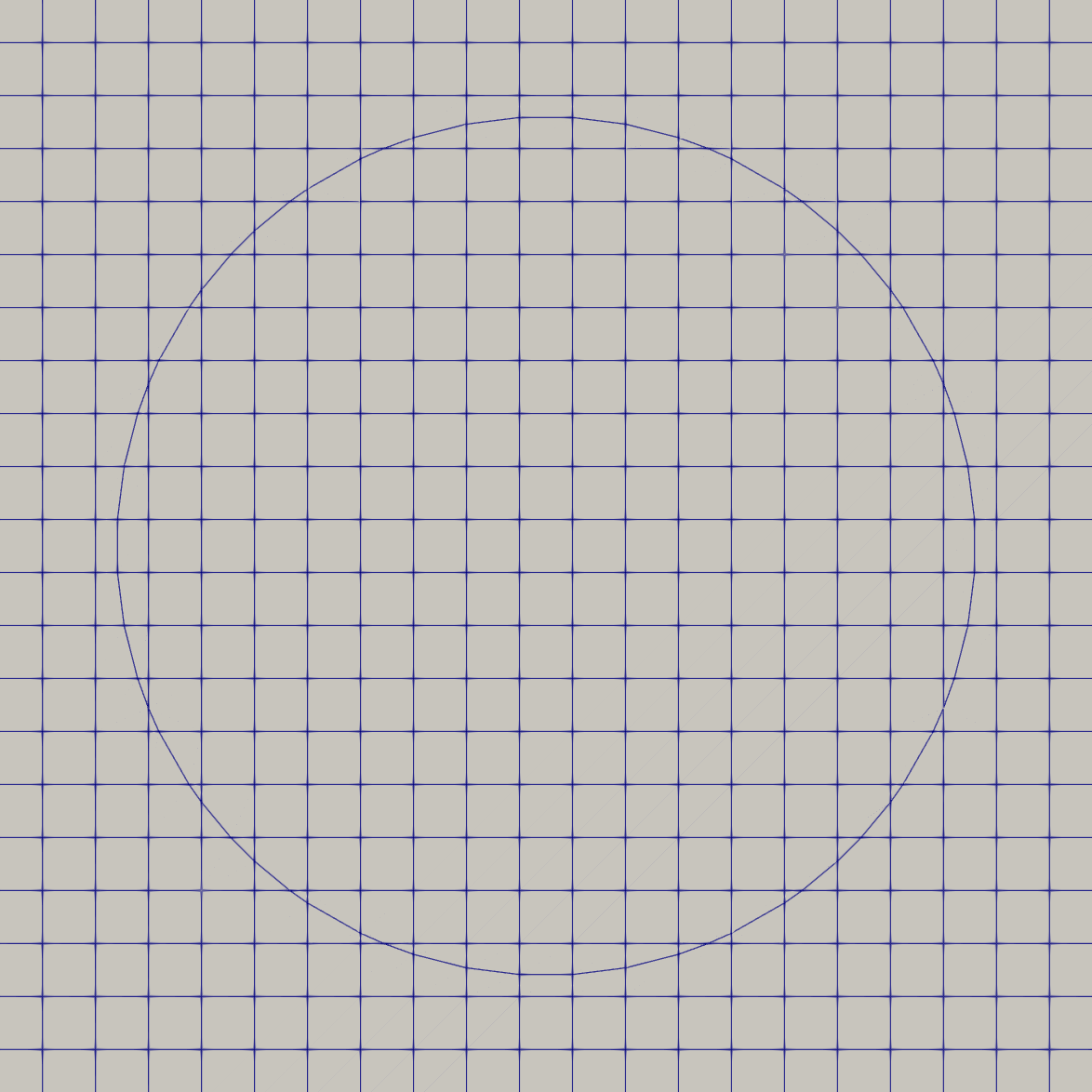}%
    \caption{Computational domain for the example in section
    \ref{subsec:numerical_example_error_decay}. On the left the domain with
    \texttt{zone 1} and \texttt{zone 2} highlighted. On the centre the
    computational grid, while on the right a zoom around the curved inclusion $\gamma$.}%
    \label{fig:ex3_grid}
\end{figure}

We solve Problem \ref{eq:semi_discr_weak_form} where we set on the top and
bottom edges a homogeneous Neumann condition, on the right edge an   absorbing
boundary condition, and on the left edge the following Dirichlet condition:

\begin{gather*}
    p(\bm{x}, t) =
    \begin{dcases*}
        \sin(\toCheck{\omega}\pi t) & if $t \leq \frac{2}{\toCheck{\omega}}$,\\
        0 & otherwise,
    \end{dcases*}
\end{gather*}
being $\omega$ the angular frequency. The initial solution and velocity are set to zero, $\rho=1$, and the final time
is set to $T=3$. In the blue region of Figure \ref{fig:ex3_grid}, outside the
inclusion $\gamma$, we fix $\mu=1$ while in the red region, inside the inclusion, we
chose $\mu = 10^{-2}$.
The source term $f$ is null, the approximation
degree set to $k=4$ and the time step equal to $\Delta t = 10^{-4}$.

We consider three different cases, depending on the wavelength \toCheck{$\lambda = \omega^{-1}$} and
the dimension of the inclusion $\gamma$. In \textit{case (i)} \toCheck{$\lambda=\frac{1}{2}$}, so the
resulting plane wave has \toCheck{a wavelength}  that is bigger than the dimension of $\gamma$. In \textit{case (ii)}, we set \toCheck{$\lambda=\frac{1}{5}$} which implies that the
\toCheck{wavelength} and the dimension of $\gamma$ are now comparable. Finally, in
\textit{case (iii)} the value of $\lambda$ is set to be \toCheck{$\lambda=\frac{1}{20}$}. We obtain a plane
wave with \toCheck{a wavelength} that is much smaller than the dimension of $\gamma$. We want
to understand, \toCheck{qualitatively}, the impact of the curved geometry in our formulation for these cases.

The results are represented in Figure \ref{fig:ex3_sol}.
\begin{figure}[p]
\centering
\begin{tabular}{ccc}
\multicolumn{3}{c}{
\text{$p_h(\textbf{x},t)$}}\\
\multicolumn{3}{c}{
\includegraphics[width=0.5\textwidth]{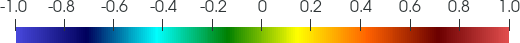}
}
\\[1.0em]
\textit{case (i)} &
\textit{case (ii)} &
\textit{case (iii)} \\[1.0em]
\includegraphics[width=0.30\textwidth]{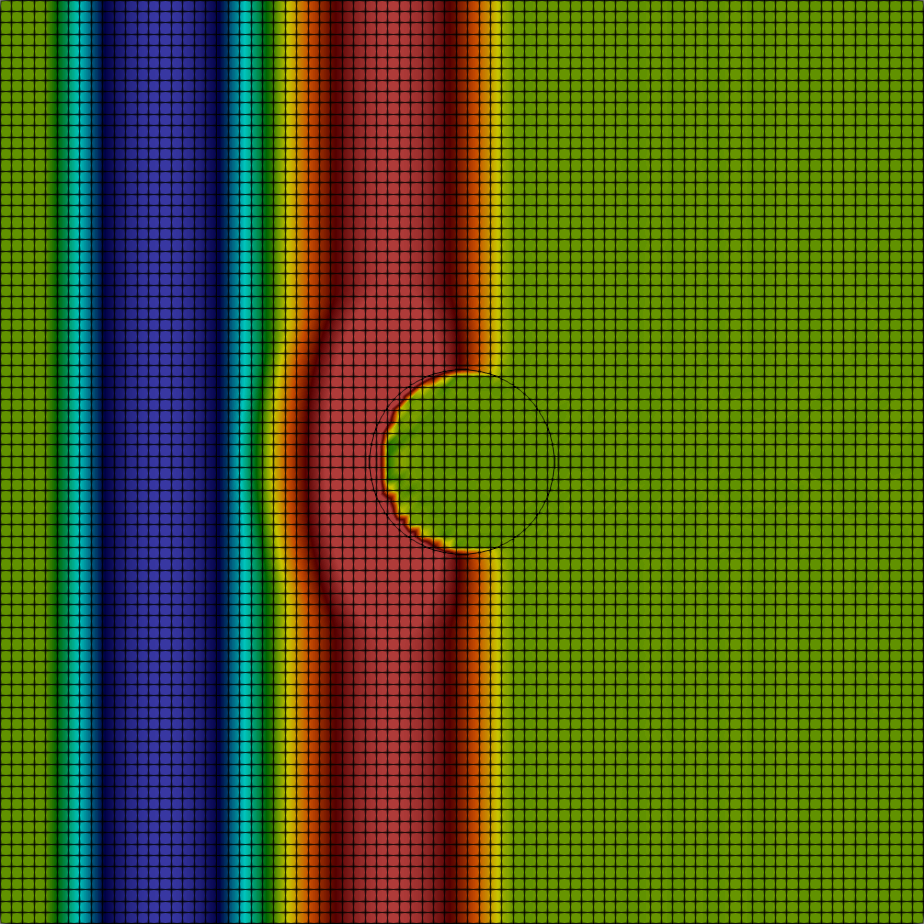} &
\includegraphics[width=0.30\textwidth]{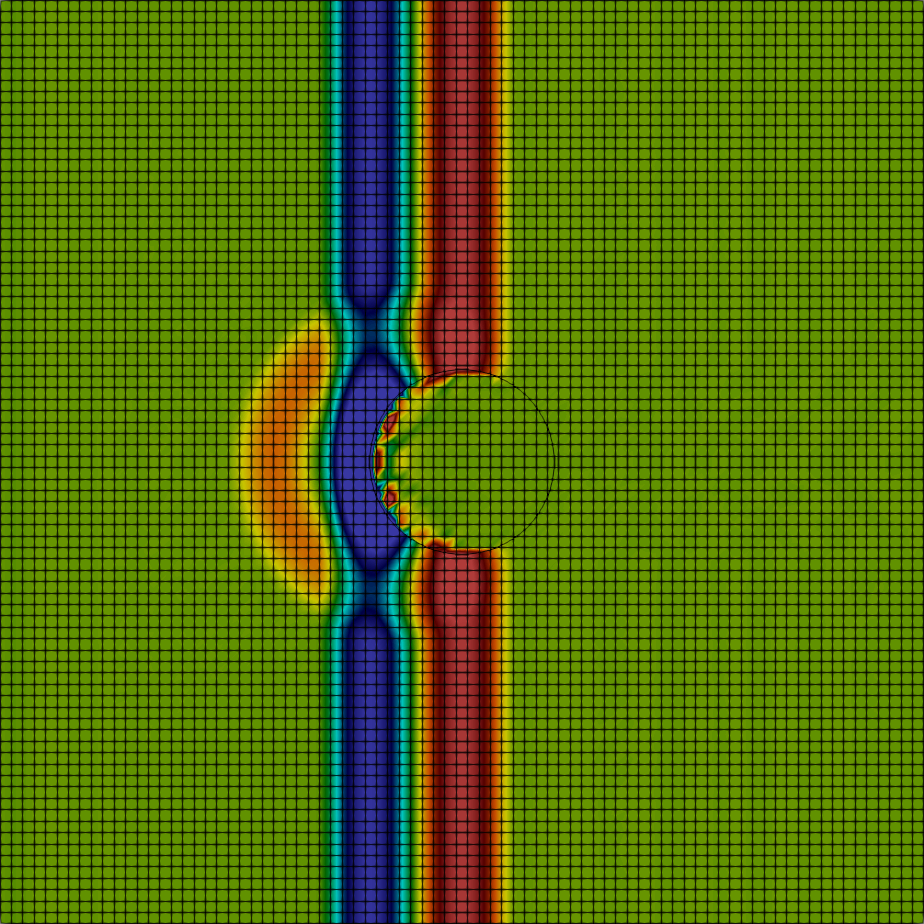} &
\includegraphics[width=0.30\textwidth]{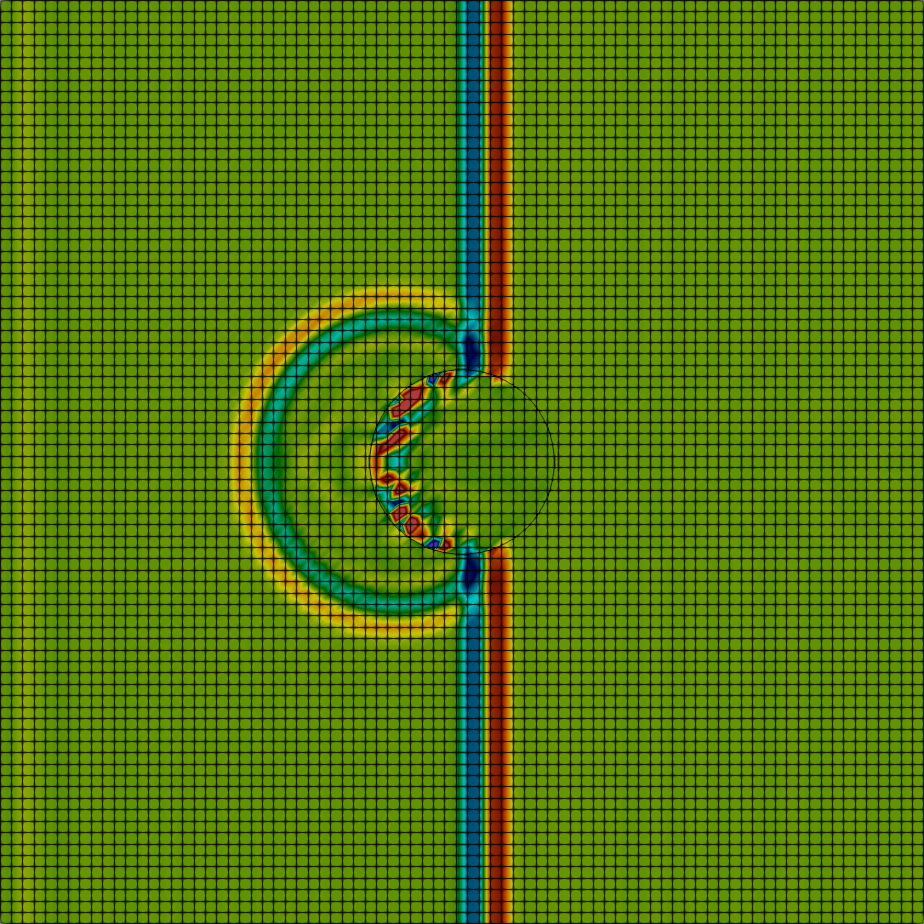} \\
$ t = 6 \Delta t$ &
$ t = 8 \Delta t$ &
$ t = 10 \Delta t$ \\[1.0em]
\includegraphics[width=0.30\textwidth]{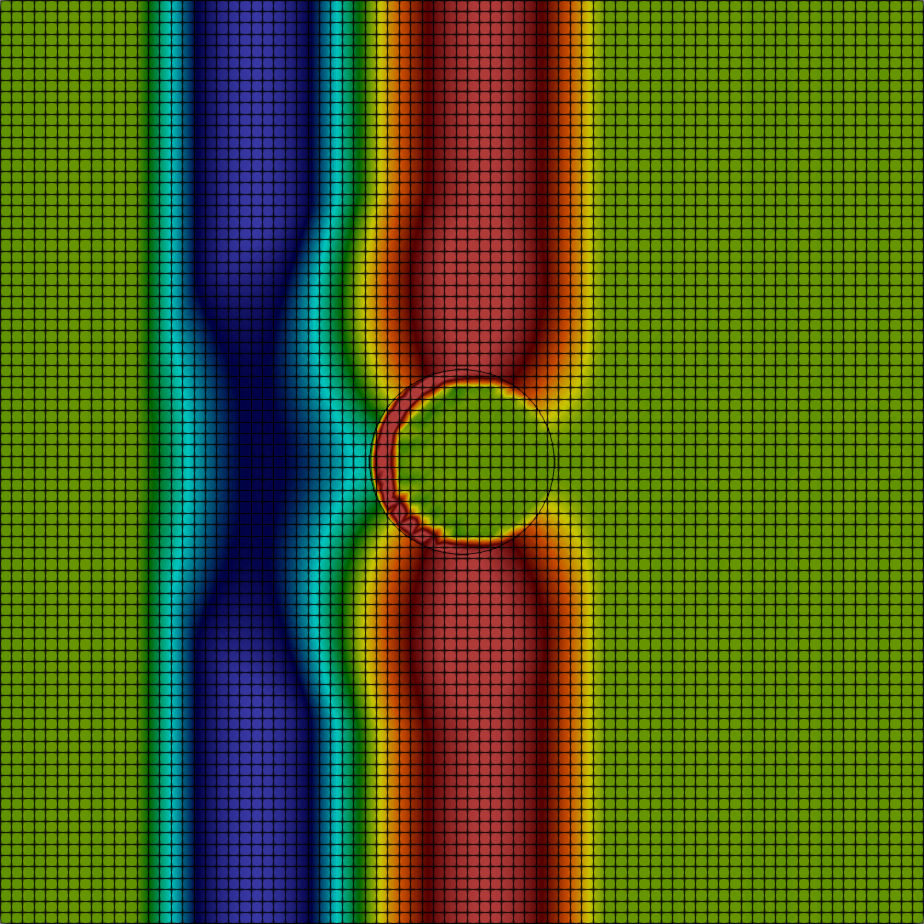} &
\includegraphics[width=0.30\textwidth]{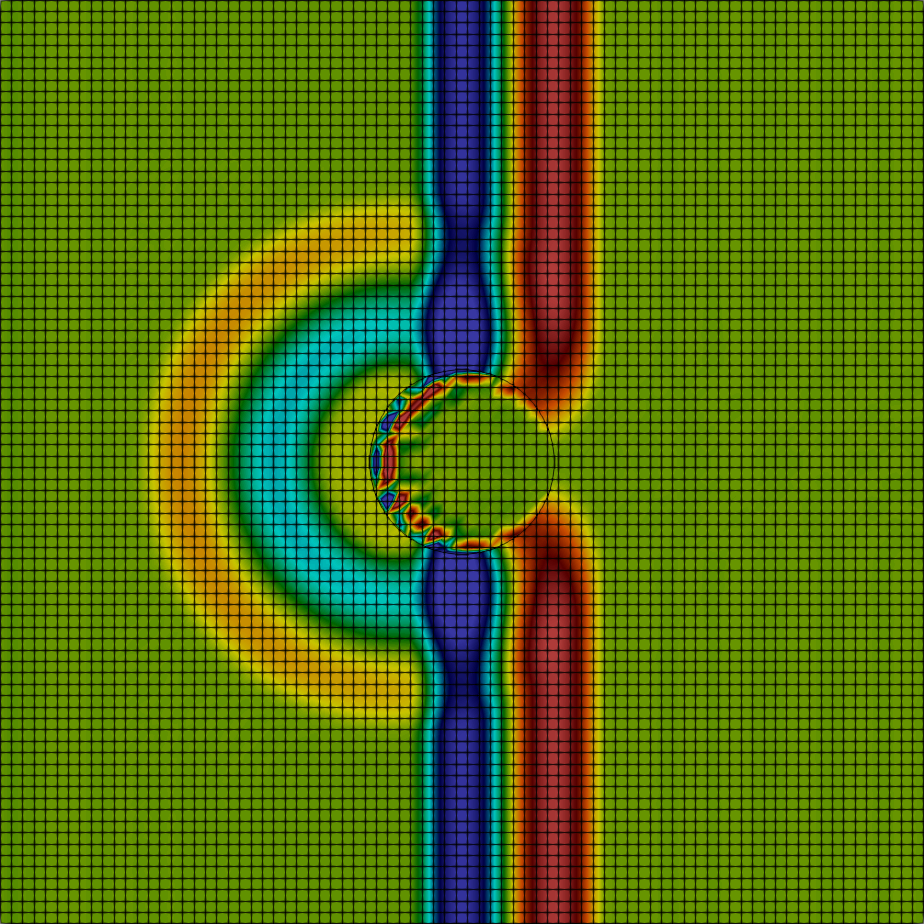} &
\includegraphics[width=0.30\textwidth]{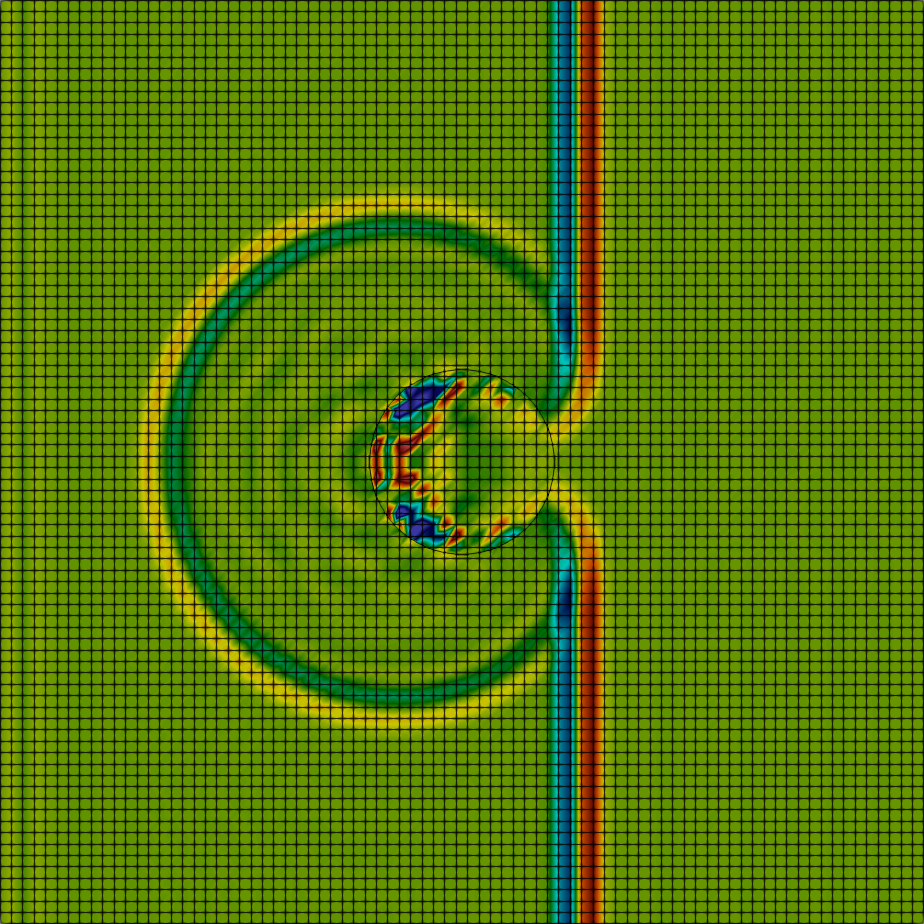} \\
$ t = 6 \Delta t$ &
$ t = 8 \Delta t$ &
$ t = 10 \Delta t$  \\[1.0em]
\includegraphics[width=0.30\textwidth]{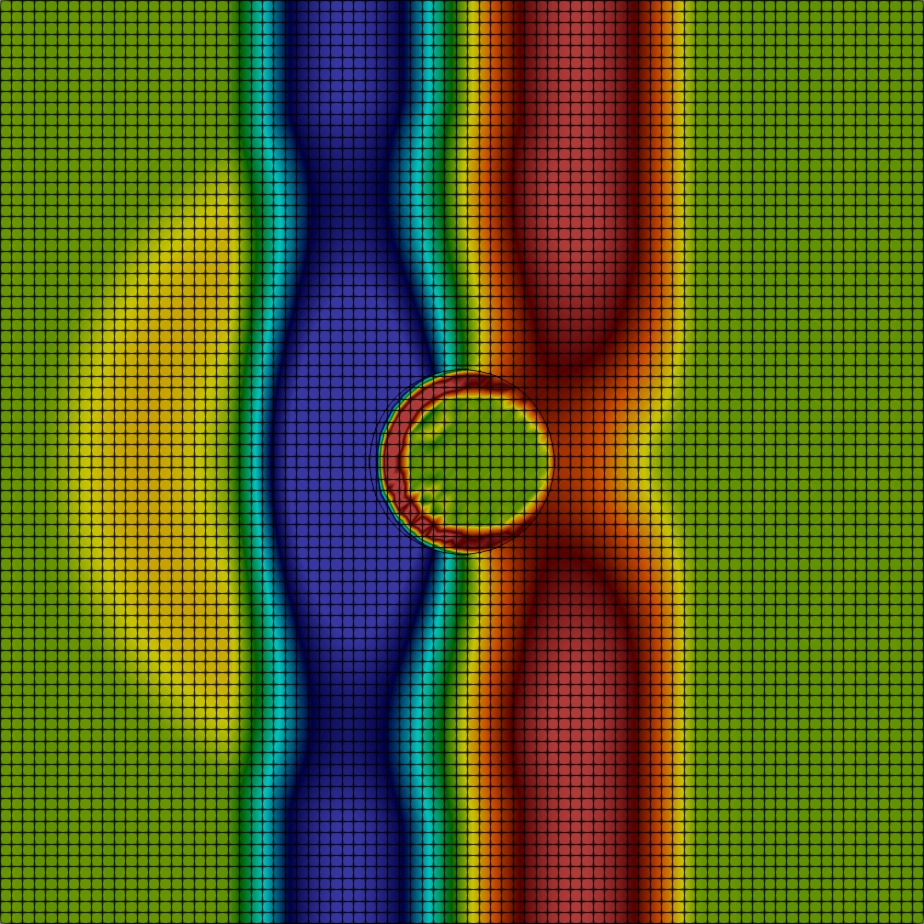} &
\includegraphics[width=0.30\textwidth]{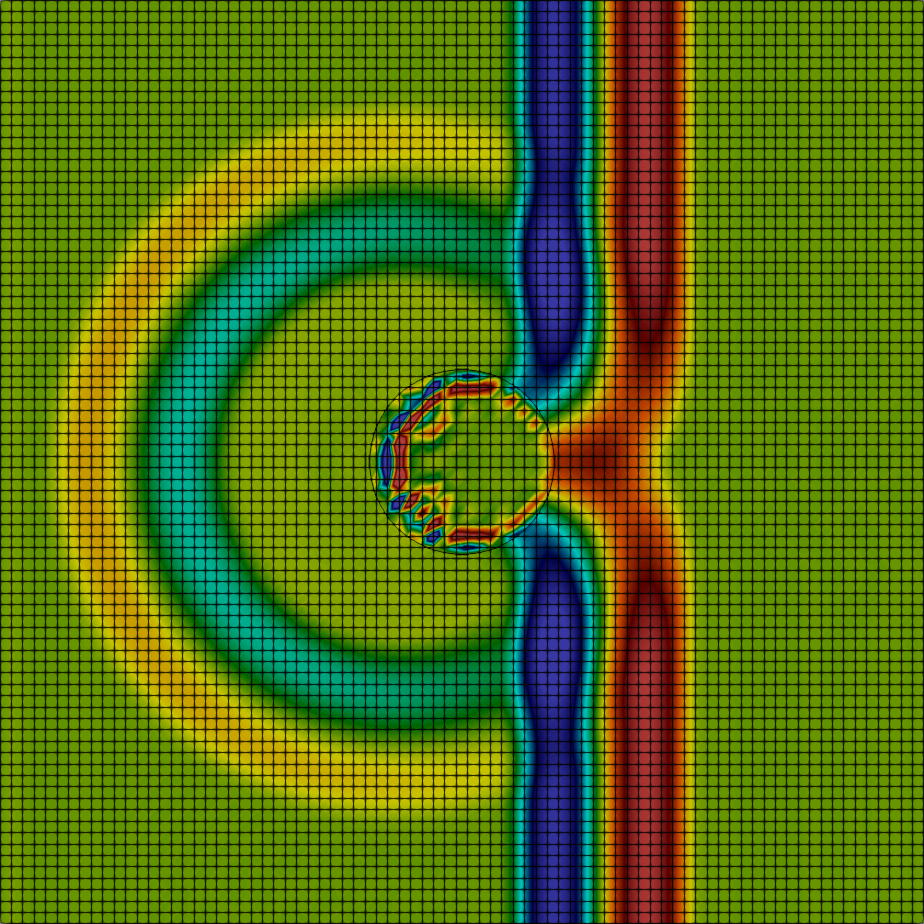} &
\includegraphics[width=0.30\textwidth]{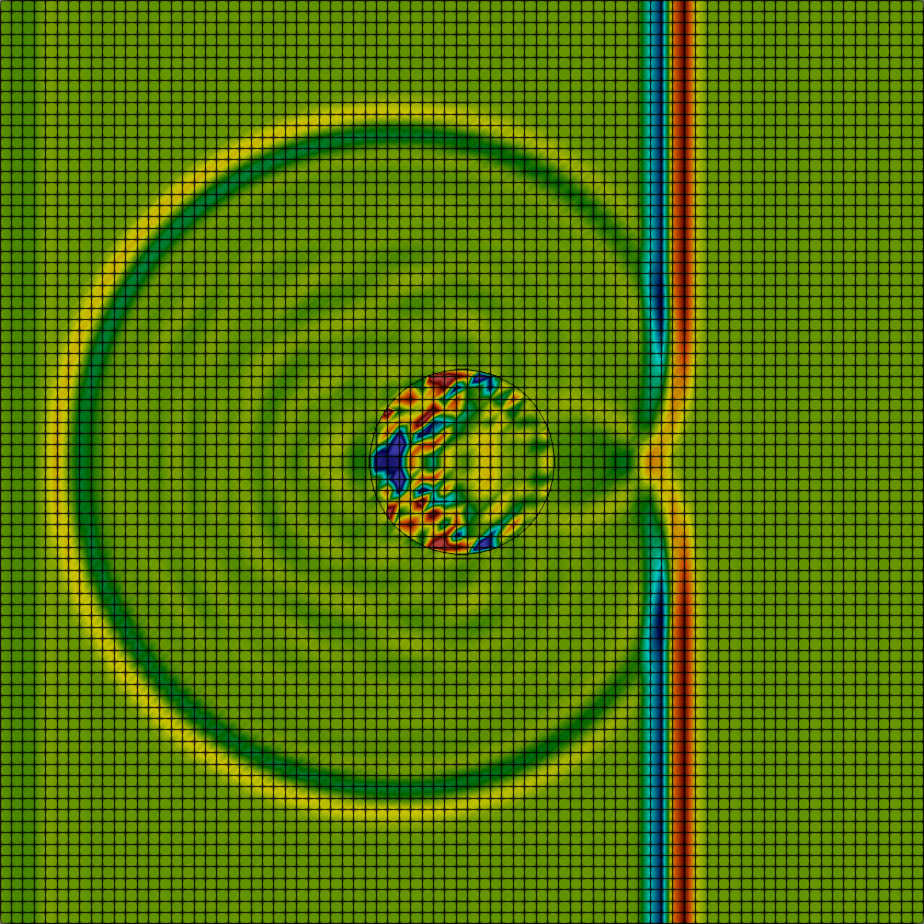}  \\
$ t = 6 \Delta t$ &
$ t = 8 \Delta t$ &
$ t = 10 \Delta t$
\end{tabular}
\caption{\toCheck{Computed solution $p_h(\textbf{x},t)$ at three different time
instants.
Left: $\lambda=\frac{1}{2}$; center: $\lambda=\frac{1}{5}$ and right: $\lambda=\frac{1}{20}$.}}
\label{fig:ex3_sol}
\end{figure}
The outcomes are, as expected, very different from each other. In \textit{case
(i)}, when the wave encounters $\gamma$ we see the two phenomena of backward
reflection of the wave and refraction inside $\gamma$. The latter is of small entity and the interior of $\gamma$ remains mostly unperturbed \toCheck{when $\lambda = \frac{1}{2}$}. This latter
phenomena is exacerbated by \toCheck{decreasing} the value of $\lambda$, indeed for
\textit{case (ii)} we notice a \toCheck{more pronounced refraction effect} inside $\gamma$ that
becomes even more \toCheck{evident} for \textit{case (iii)}.
This phenomena are expected and confirm the quality of the obtained solution.
No spurious oscillations \as{due to the geometry} (at least macroscopic ones) can be noticed in the reported plots.

We can conclude that, for this example, the proposed method is
an attractive approach to solve the wave equation with high order approximation
in presence of curved interfaces without the need of refining the computational
\fd{grid nearby them.}

\subsection{Wave propagation with a realistic curved geometry}\label{subsec:numerical_example_listric_fault}

In this last test case, we consider a curved geometry that might represent a
realistic case of a listric fault \as{cutting a sequence of sedimentary layers}. However, \as{for the sake of simplicity}, its
dimensions are  set as $\Omega = (-1, 1) \times (-0.5, 0.5)$, see Figure~\ref{fig:ex4_geo}.
\begin{figure}[tbp]
    \centering
    \resizebox{0.45\textwidth}{!}{\fontsize{10cm}{10cm}\selectfont
\begingroup%
  \makeatletter%
  \providecommand\color[2][]{%
    \errmessage{(Inkscape) Color is used for the text in Inkscape, but the package 'color.sty' is not loaded}%
    \renewcommand\color[2][]{}%
  }%
  \providecommand\transparent[1]{%
    \errmessage{(Inkscape) Transparency is used (non-zero) for the text in Inkscape, but the package 'transparent.sty' is not loaded}%
    \renewcommand\transparent[1]{}%
  }%
  \providecommand\rotatebox[2]{#2}%
  \newcommand*\fsize{\dimexpr\f@size pt\relax}%
  \newcommand*\lineheight[1]{\fontsize{\fsize}{#1\fsize}\selectfont}%
  \ifx\svgwidth\undefined%
    \setlength{\unitlength}{483.65363174bp}%
    \ifx\svgscale\undefined%
      \relax%
    \else%
      \setlength{\unitlength}{\unitlength * \real{\svgscale}}%
    \fi%
  \else%
    \setlength{\unitlength}{\svgwidth}%
  \fi%
  \global\let\svgwidth\undefined%
  \global\let\svgscale\undefined%
  \makeatother%
  \begin{picture}(1,0.53809576)%
    \lineheight{1}%
    \setlength\tabcolsep{0pt}%
    \put(0,0){\includegraphics[width=\unitlength,page=1]{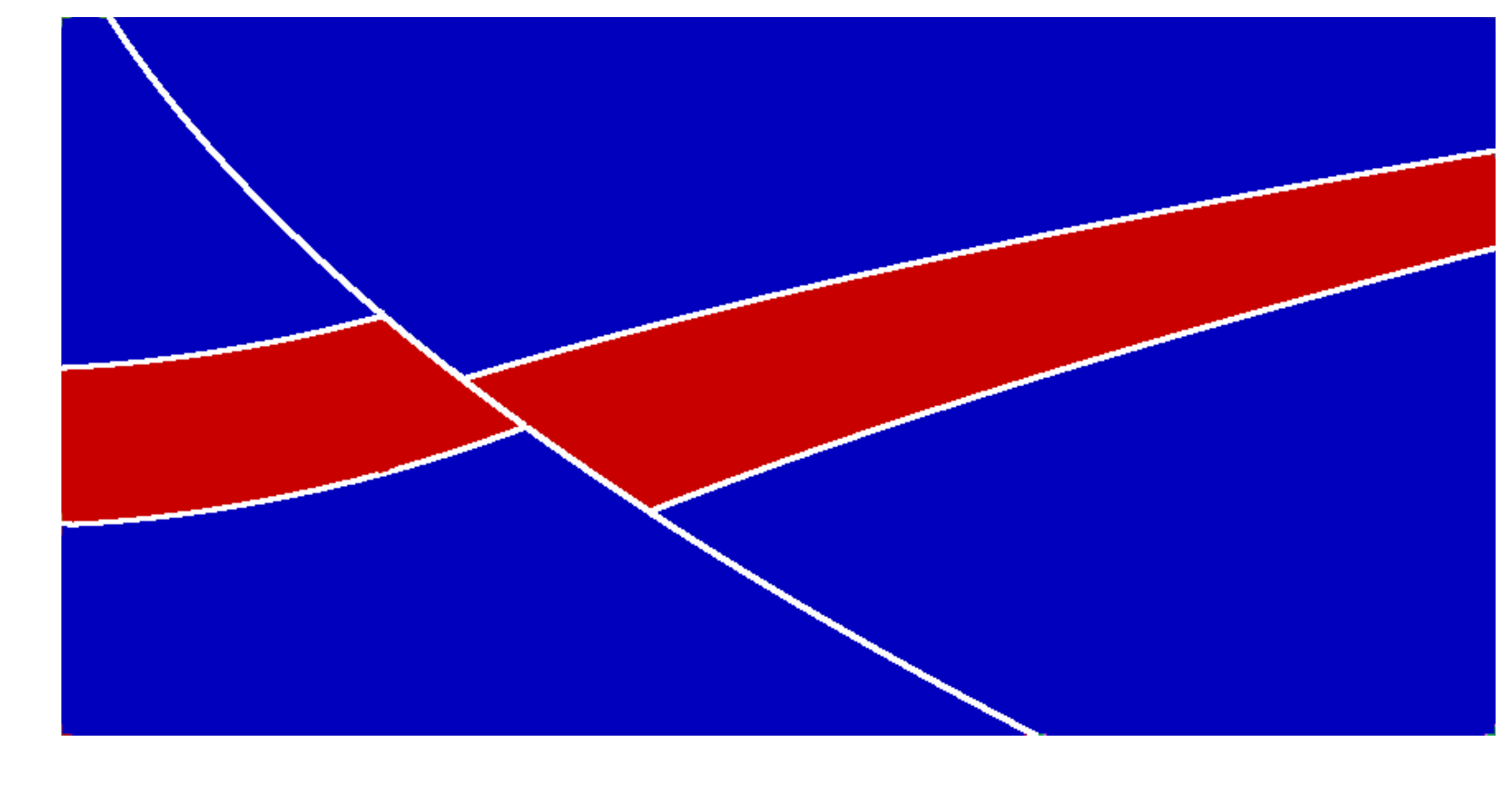}}%
    \put(0.08463671,0.24509962){\color[rgb]{0,0,0.06666667}\makebox(0,0)[lt]{\lineheight{1.25}\smash{\begin{tabular}[t]{l}$\Omega_{\rm mid}$\end{tabular}}}}%
    \put(0.60232342,0.32018722){\color[rgb]{0,0,0.06666667}\makebox(0,0)[lt]{\lineheight{1.25}\smash{\begin{tabular}[t]{l}$\Omega_{\rm mid}$\end{tabular}}}}%
    \put(0.42348025,-0.03233764){\color[rgb]{0,0,0}\makebox(0,0)[lt]{\lineheight{1.25}\smash{\begin{tabular}[t]{l}$x$\end{tabular}}}}%
    \put(-0.04725406,0.44339395){\color[rgb]{0,0,0}\makebox(0,0)[lt]{\lineheight{1.25}\smash{\begin{tabular}[t]{l}$y$\end{tabular}}}}%
    \put(0,0){\includegraphics[width=\unitlength,page=2]{listric.pdf}}%
  \end{picture}%
\endgroup%
}
    \caption{Computational domain for the example in section
    \ref{subsec:numerical_example_listric_fault}.}
    \label{fig:ex4_geo}
\end{figure}
In the middle red areas are defined as $\Omega_{\rm mid}$. 
This central layer has different physical parameters than the surrounding portions of materials, see blue areas in Figure~\ref{fig:ex4_geo}.
Even if realistic applications with multiple layers are more
challenging \as{than this example}, we show a qualitative analysis to understand the \as{potentiality} of
the newly introduced method.  Then, we consider Problem~\eqref{modelpb}, with $\rho=1$ everywhere,
$\mu=0.1$ in $\Omega_{\rm mid}$ and $\mu=1$ everywhere else. On all the
boundaries we impose absorbing conditions, the initial data are null, the final time is $T=10$ and the
source term is given as \toCheck{a Mexican-hat wavelet}
\begin{gather*}
    f(\bm{x}, t) =
    5[1-2\pi^2(t-1.2)^2]e^{-\pi^2(t-1.2)}e^{-\frac{(x_0+0.375)^2+(x_1-0.25)^2}{0.00625}}.
\end{gather*}

The geometry is obtained from a $32 \time 64$ Cartesian grid where the curved elements are thus created by cutting the mesh with the curved interfaces. The resulting computational grid has 2271 elements. We consider a polynomial approximation order equal to $4$ and a time step $\Delta t = 0.001$.  The obtained
numerical solution is depicted in Figure \ref{fig:ex4_sol} for different time
instants.
\begin{figure}[p]
    \centering
    \begin{tabular}{cc}
        \multicolumn{2}{c}{\text{$p_h(\textbf{x},t)$}}\\
        \multicolumn{2}{c}{\includegraphics[width=0.45\textwidth]{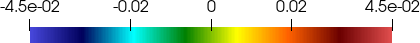}}
        \\[1.0em]
        \includegraphics[width=0.45\textwidth]{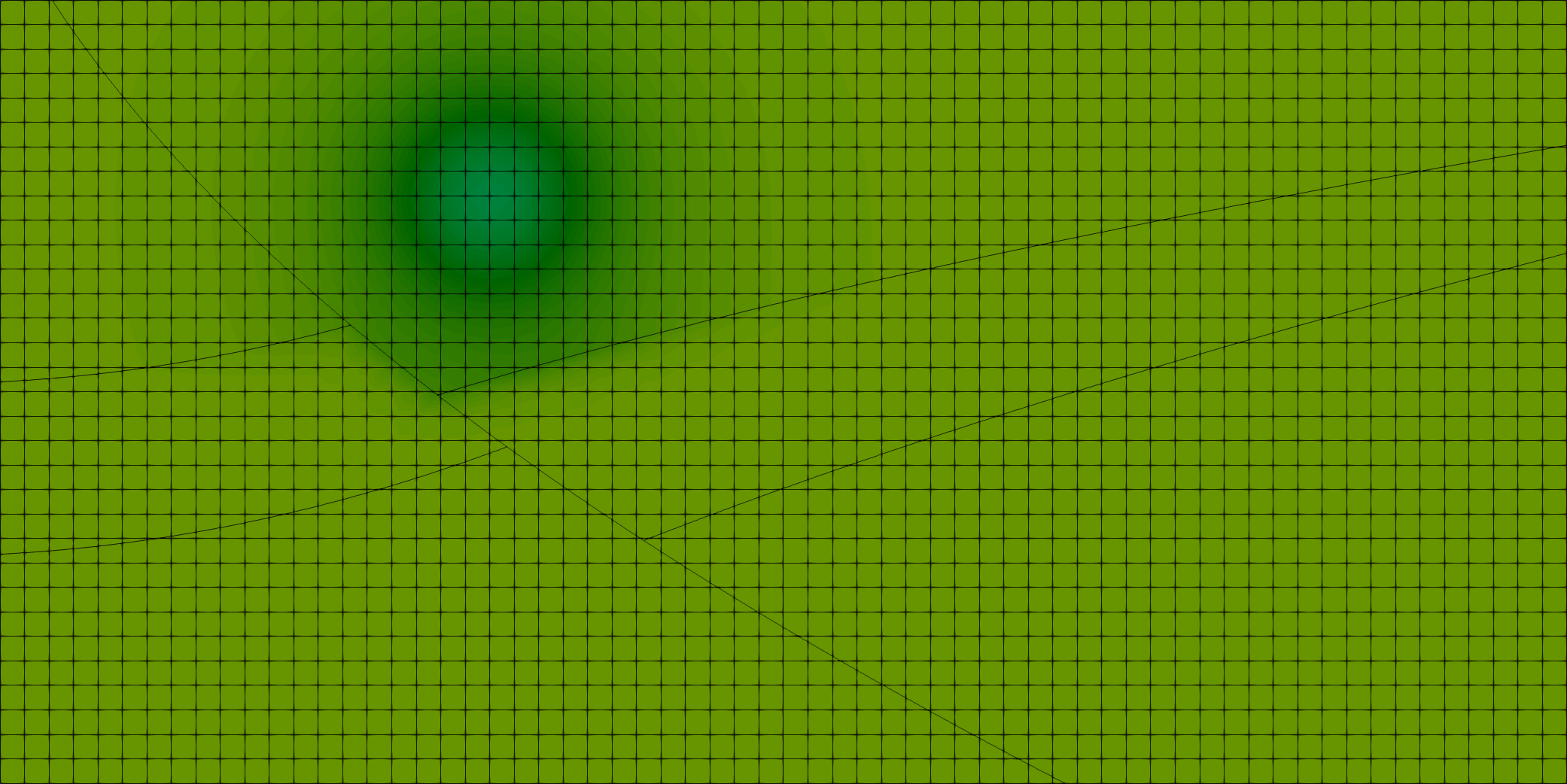}
        &
        \includegraphics[width=0.45\textwidth]{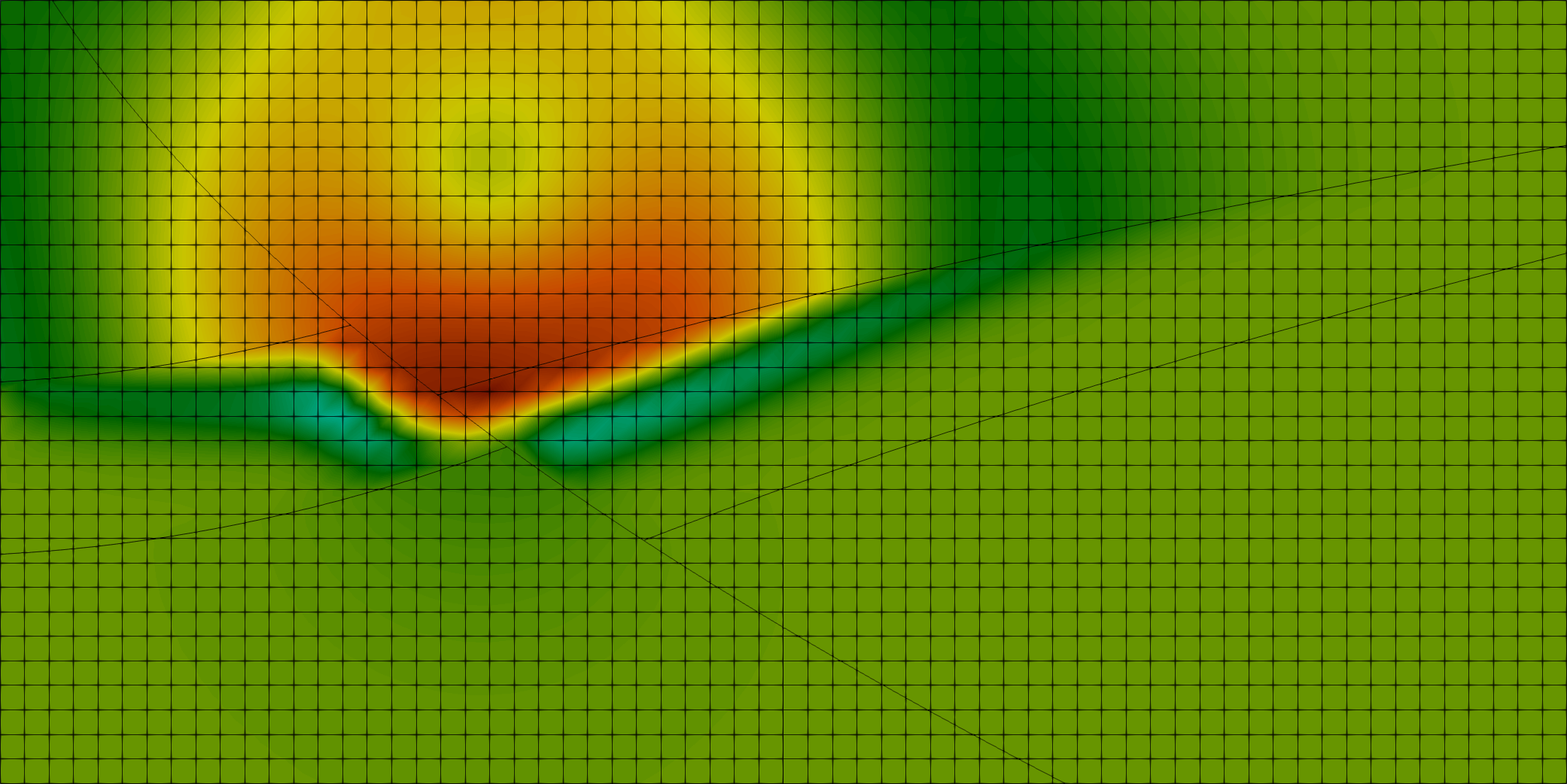}\\
        $ t = \Delta t$ &
        $ t = 5 \Delta t$
        \\[1.0em]
        \includegraphics[width=0.45\textwidth]{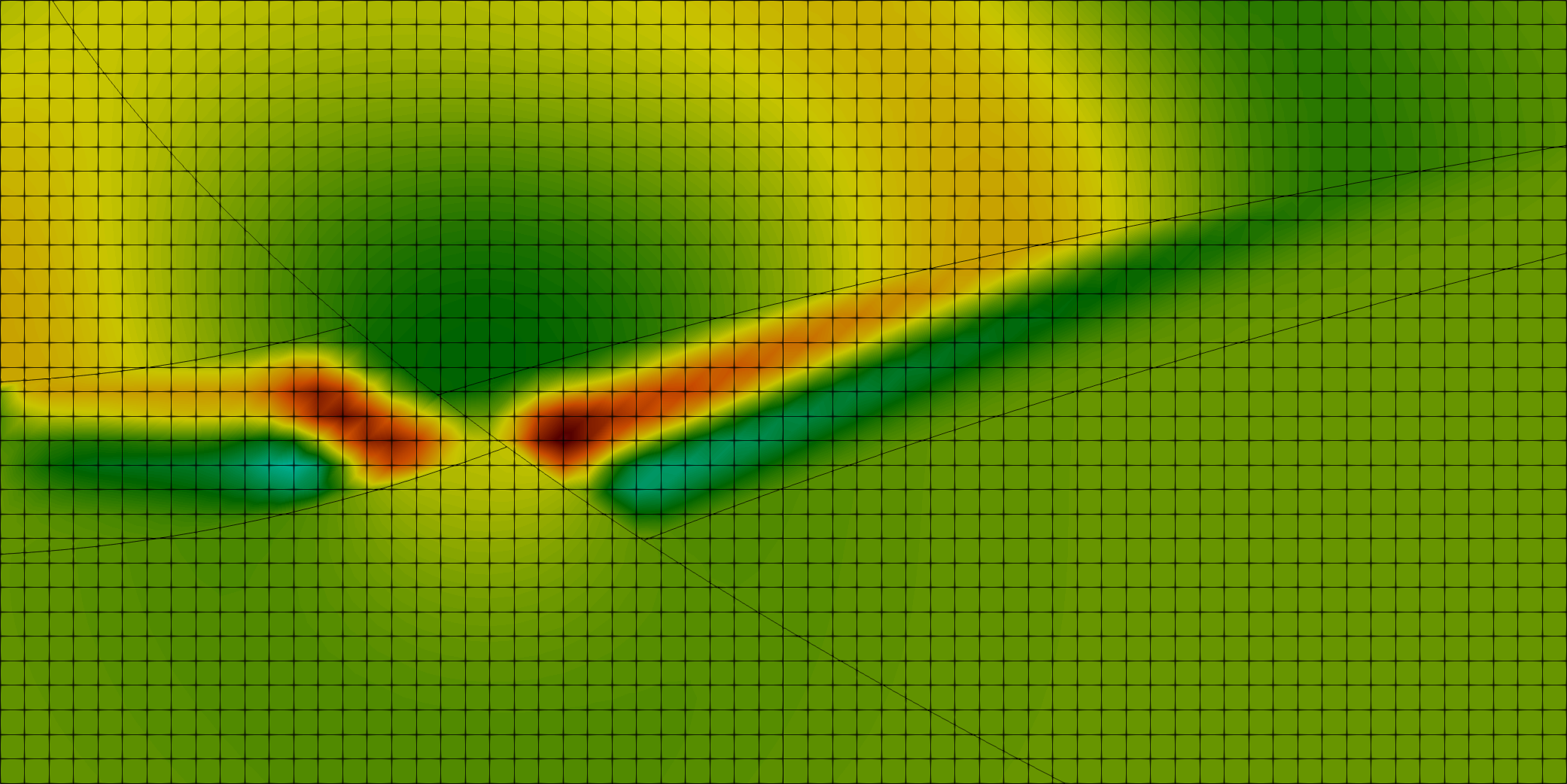}%
        &
        \includegraphics[width=0.45\textwidth]{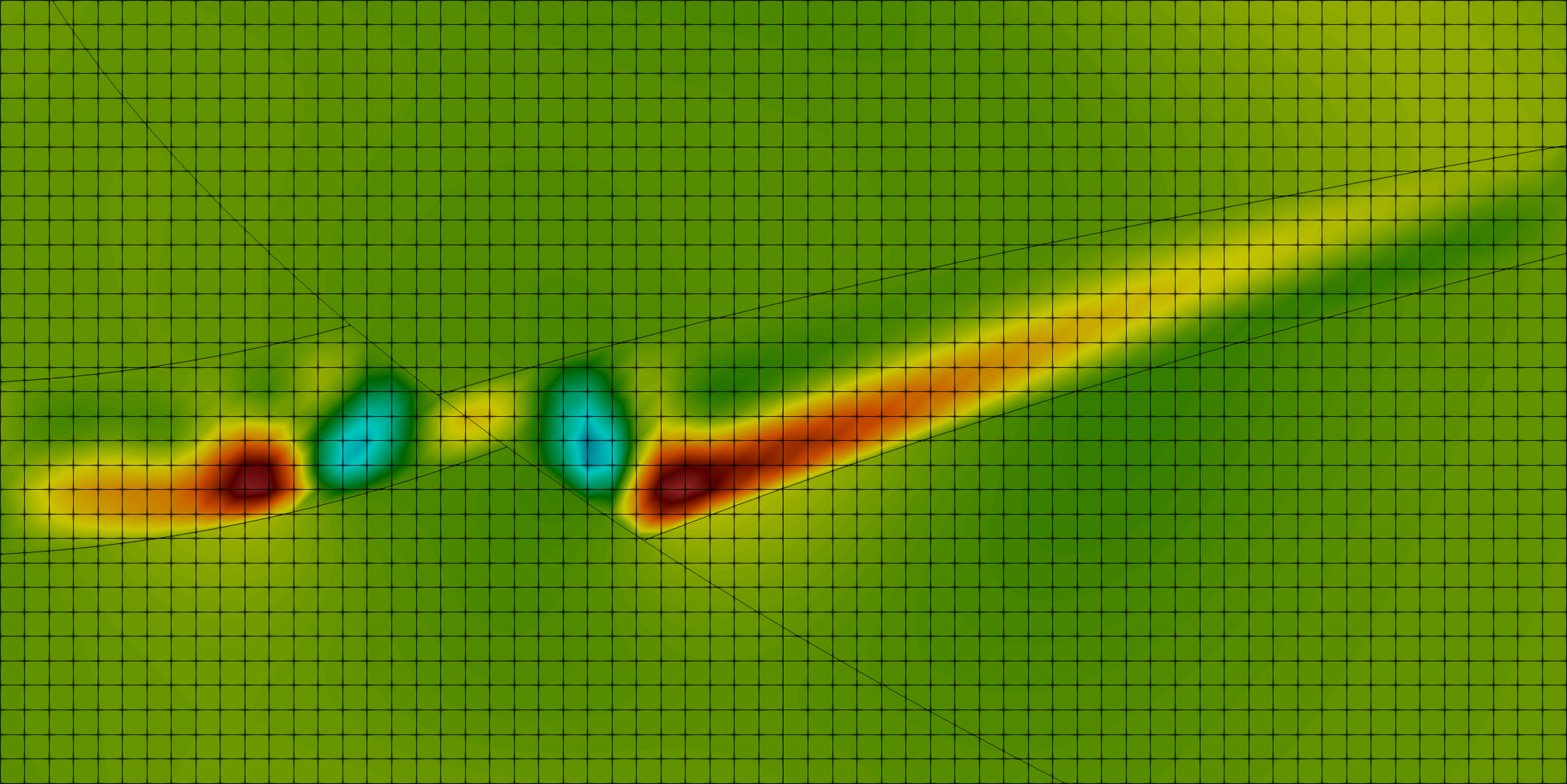}\\
        $t  = 7 \Delta t$ &
        $t  = 10 \Delta t$
        \\[1.0em]
        \includegraphics[width=0.45\textwidth]{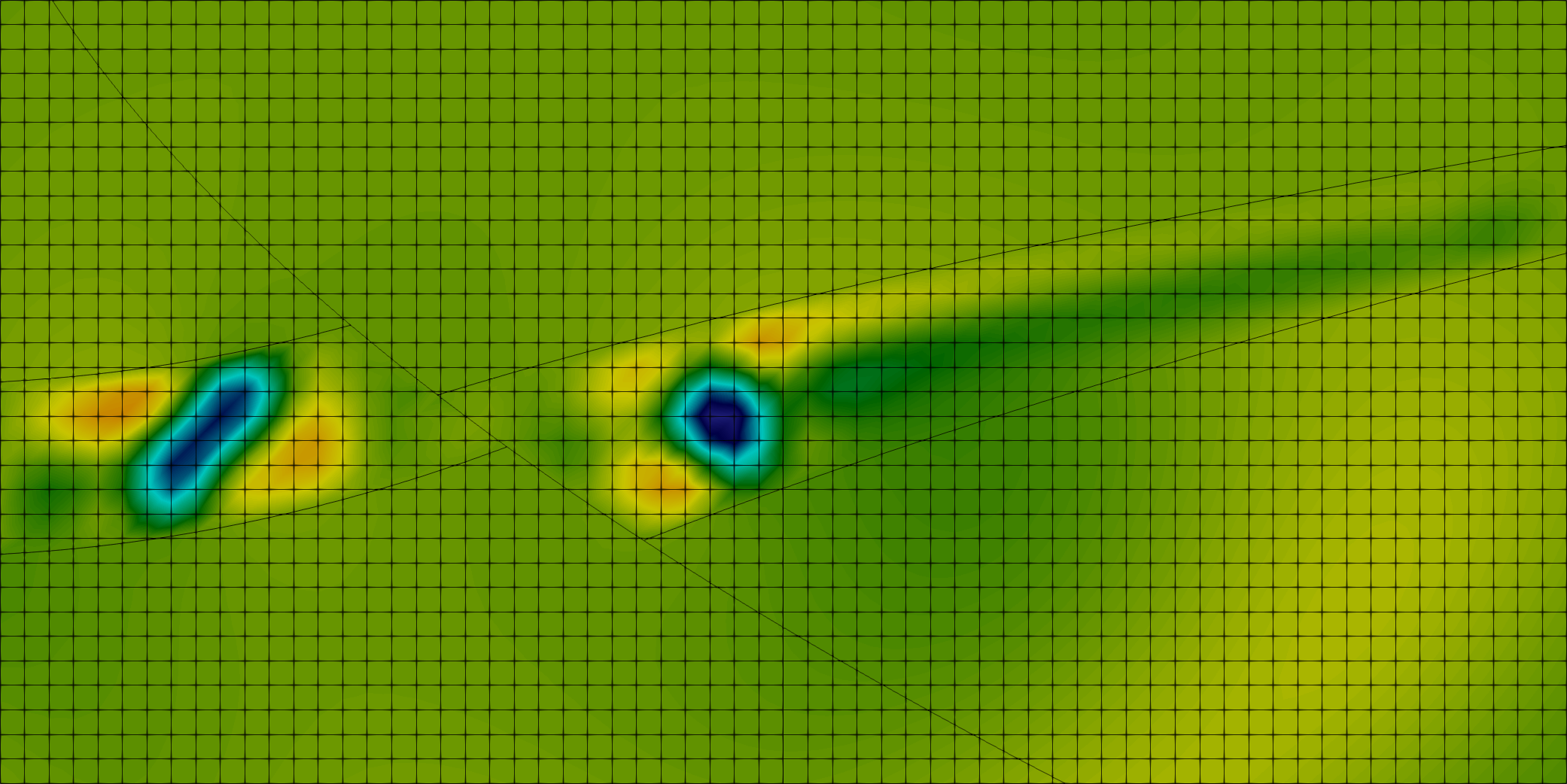}%
        &
        \includegraphics[width=0.45\textwidth]{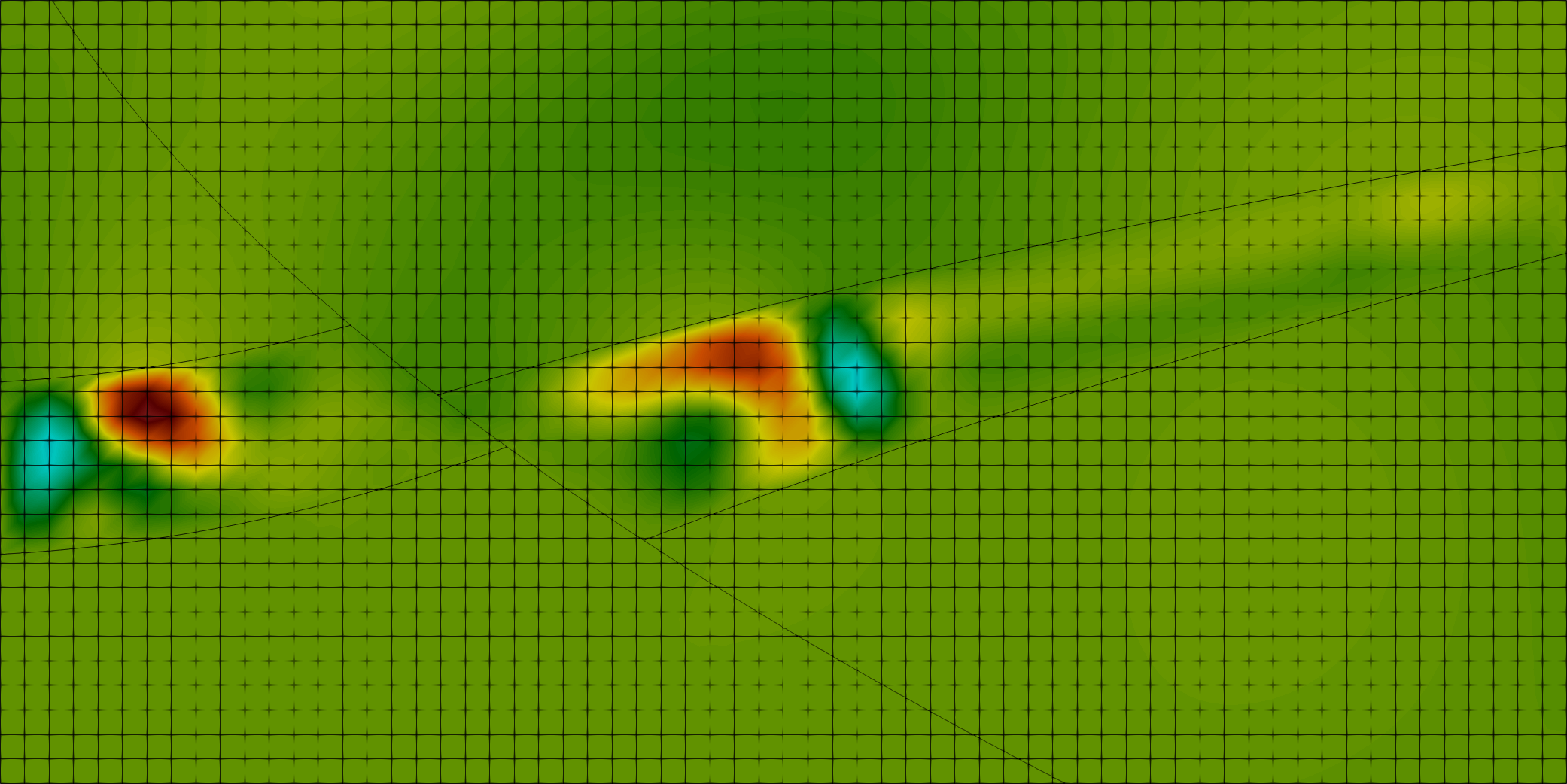}\\
        $t  = 14 \Delta t$ &
        $t  = 18 \Delta t$
        \\[1.0em]
        \includegraphics[width=0.45\textwidth]{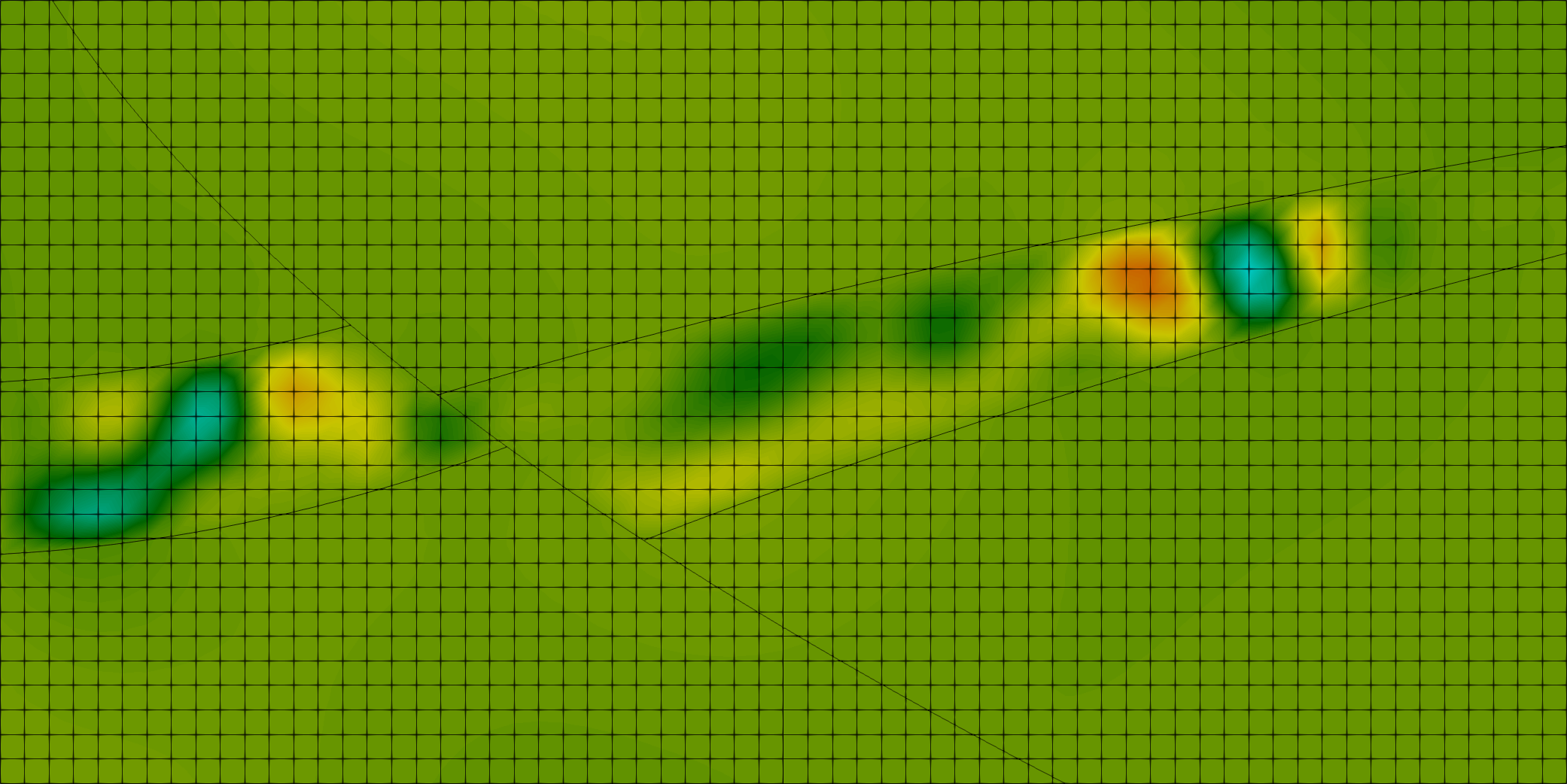}%
        &
        \includegraphics[width=0.45\textwidth]{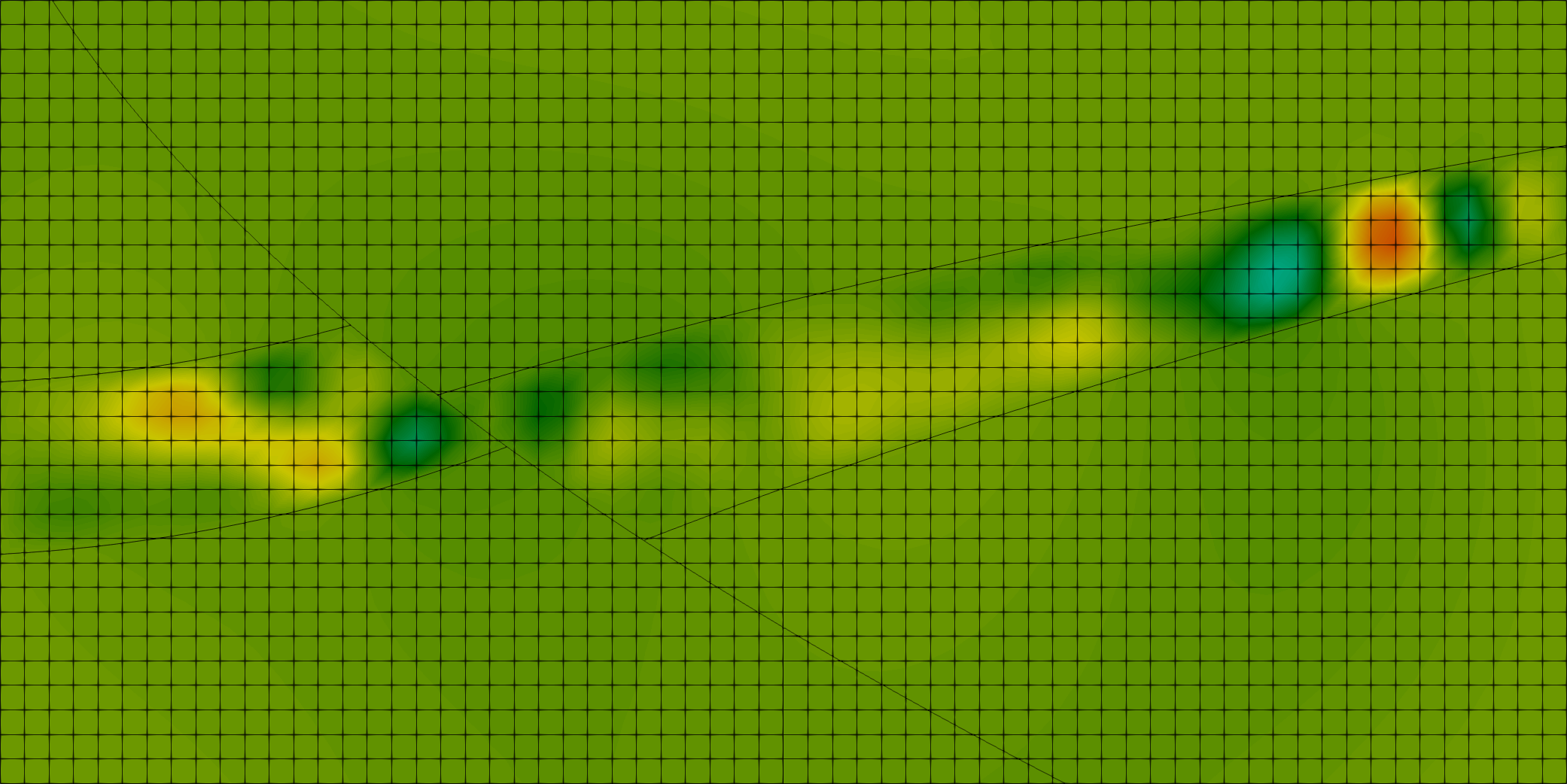}\\
        $t  = 34 \Delta t$ &
        $t  = 42 \Delta t$
    \end{tabular}
    \caption{From left to right and from top to bottom, screenshots of the computed
    solution $p_h(\textbf{x},t)$ at different time instants.}%
    \label{fig:ex4_sol}
\end{figure}

From the obtained solution, we see that no spurious oscillations are generated at the curved interfaces inside the domain. Moreover, since  the characteristic velocity in $\Omega_{\rm mid}$ is smaller than the surrounding parts, the wave remains trapped inside $\Omega_{\rm mid}$.

Also in this final test case, even on a more complex curved geometry, the proposed scheme performs well, without any need to over-refine
around the curved interfaces to avoid side effects.

\section{Conclusions}\label{Sc:Conclusions}

\toCheck{In this paper, we have extended the Virtual Element Method in primal form for the wave equation
when the computational domain has internal curved interfaces and/or curved boundaries. The latter are represented exactly in order to avoid possible geometrical errors that
might affect the quality of the numerical solution and limit the convergence order.
This preliminary study carried out in a  two-dimensional setting and 
the promising results shown, open the possibility to extend the proposed approach to more complex three dimensional configurations. 
The numerical
examples presented testify that this approach is very effective, since no spurious oscillations due to curved geometries arise. Moreover, it gives  the possibility to handle the geometrical challenges arising in
realistic applications.}

\section*{Aknowledgements}
The authors are members of the INdAM Research group GNCS and this work is partially funded by INdAM-GNCS through the project ``Bend VEM 3d''.

\bibliographystyle{abbrv}
\bibliography{biblio}

\end{document}